\documentclass[12pt,a4paper]{amsart}
\usepackage{amsmath}
\usepackage{amsthm}
\usepackage{enumerate}
\usepackage{amssymb}
\usepackage[utf8]{inputenc}
\usepackage{color}
\usepackage{tikz}
\usepackage{verbatim}
\usepackage{hyperref,aliascnt}
\usepackage[all,cmtip,arrow,matrix,curve]{xy}

\theoremstyle{plain}

\newtheorem{lemma}{Lemma}[section]

\newaliascnt{thmCt}{lemma}
\newtheorem{theorem}[thmCt]{Theorem}
\aliascntresetthe{thmCt}

\newaliascnt{corCt}{lemma}
\newtheorem{corollary}[corCt]{Corollary}
\aliascntresetthe{corCt}

\newaliascnt{prpCt}{lemma}
\newtheorem{proposition}[prpCt]{Proposition}
\aliascntresetthe{prpCt}

\newtheorem*{proposition*}{Proposition}
\newtheorem*{definition*}{Definition}
\newtheorem*{theorem*}{Theorem}

\theoremstyle{definition}

\newaliascnt{dfnCt}{lemma}

\aliascntresetthe{dfnCt}

\newaliascnt{rmkCt}{lemma}
\newtheorem{remark}[rmkCt]{Remark}
\aliascntresetthe{rmkCt}

\newaliascnt{rmksCt}{lemma}

\aliascntresetthe{rmksCt}

\newaliascnt{qstCt}{lemma}

\aliascntresetthe{qstCt}

\newaliascnt{exaCt}{lemma}
\newtheorem{example}[exaCt]{Example}
\aliascntresetthe{exaCt}

\newaliascnt{exasCt}{lemma}

\aliascntresetthe{exasCt}

\newaliascnt{paraCt}{lemma}
\newtheorem{parag}[paraCt]{}
\aliascntresetthe{paraCt}

\newcounter{theoremintro}
\newtheorem{thmintro}[theoremintro]{Theorem}

\newcommand{\id}{\mathrm{id}}

\newcommand{\POM}{\mathrm{PoM}}

\newcommand{\preS}{{\mathcal S}}
\newcommand{\SR}{{\mathrm S}}

\newcommand{\CP}{\mathrm{CP}}

\newcommand{\W}{{\mathrm W}}
\newcommand{\V}{{\mathrm V}}

\newcommand{\K}{{\mathrm K}}

\newcommand{\Cu}{\mathrm{Cu}}
\newcommand{\SCu}{\mathrm{SCu}}
\newcommand{\Q}{\mathcal{Q}}

\newcommand{\SQ}{\mathrm{S}\mathcal{Q}}

\newcommand{\Tr}{\mathrm{Tr}}

\newcommand{\Z}{{\mathbb Z}}
\newcommand{\N}{{\mathbb N}}
\newcommand{\R}{{\mathbb R}}
\newcommand{\ol}{\overline}

\title[Ideals, quotients, and continuity]{Ideals, quotients, and continuity of the Cuntz semigroup for rings}
\date{\today}

\author[R.~Antoine, P.~Ara, J.~Bosa, F.~Perera, and E.~Vilalta]{Ramon Antoine\and 
        Pere Ara\and
        Joan Bosa\and
        Francesc Perera\and
        Eduard Vilalta}
        
\address{Ramon Antoine, 
	Departament de Matem\`{a}tiques,
	Universitat Aut\`{o}noma de Bar\-ce\-lo\-na,
	08193 Bellaterra, Barcelona, Spain, and
	Centre de Recerca Matem\`atica, Edifici Cc, Campus de Bellaterra, 08193 Cerdanyola del Vall\`es, Barcelona, Spain}
\email{ramon.antoine@uab.cat}

\address{Pere Ara, 
		Departament de Matem\`{a}tiques,
		Universitat Aut\`{o}noma de Barcelona,
		08193 Bellaterra, Barcelona, Spain, and
		Centre de Recerca Matem\`atica, Edifici Cc, Campus de Bellaterra, 08193 Cerdanyola del Vall\`es, Barcelona, Spain}
\email{pere.ara@uab.cat}

\address{Joan Bosa,
Departamento de Matem\'{a}ticas,
Universidad de Zaragoza,
50009 Zaragoza, Zaragoza, Spain.}
\email{jbosa@unizar.es}
\urladdr{personal.unizar.es/jbosa/}

\address{Francesc Perera,
Departament de Matem\`{a}tiques,
Universitat Aut\`{o}noma de Barcelona,
08193 Bellaterra, Barcelona, Spain, and
Centre de Recerca Matem\`a\-ti\-ca, Edifici Cc, Campus de Bellaterra,  08193 Cerdanyola del Vall\`es, Barcelona, Spain}
\email[]{francesc.perera@uab.cat}
\urladdr{https://mat.uab.cat/web/perera}

\address{Eduard Vilalta, 
Department of Mathematical Sciences, Chalmers University of Technology and University of Gothenburg, Chalmers Tvargata 3, SE-412 96 Gothenburg, Sweden}
\email[]{vilalta@chalmers.se}
\urladdr{www.eduardvilalta.com}

\thanks{All authors were partially supported by the Spanish Research State Agency (grant No.\  PID2020-113047GB-I00/AEI/10.13039/501100011033) and by the Comissionat per Universitats i Recerca de la Ge\-ne\-ralitat de Ca\-ta\-lu\-nya (grant No.\ 2021-SGR-01015).~The third named author was also partially supported by the Spanish State Research Agency through Consolidación Investigadora program No.~CNS2022-135340 and by the Dep.~ciencia, universidad y sociedad del conocimiento del Gobierno de Aragón (grupo E22-23R). The last named author was also supported by the Fields Institute for Research in Mathematical Sciences and by the Knut and Alice Wallenberg Foundation (KAW 2021.0140).~This work is also partially supported by the Spanish State Research Agency, through the Severo Ochoa and María de Maeztu Program for Centers and Units of Excellence in R\&D (CEX2020-001084-M)}

\subjclass[2020]%
{Primary
	16D25 
	06F05; 
    16D10, 
	16B99, 
	Secondary
	46L05. 
}

\begin{document}

\begin{abstract}
 In this paper we explore which part of the ideal lattice of a general ring is parametrized by its Cuntz semigroup $\SR(R)$ and its ambient semigroup $\Lambda(R)$. We identify these classes of ideals as the quasipure ideals (a generalization of pure ideals) in the case of $\SR(R)$, and what we term decomposable ideals in the case of $\Lambda(R)$. For an ($s$-)unital ring $R$, the latter class exhausts all ideals of the ring. We prove that these constructions behave well with respect to quotients. In order to study the passage to inductive limits, we introduce the classes of dense and left normal rings. We show that $\SR(R)$ is an abstract $\Cu$-semigroup whenever $R$ is left normal and, for such rings, the assignment $R\mapsto \SR(R)$ is continuous. We prove a parallel result for $\Lambda(R)$ whenever $R$ is a dense ring.
\end{abstract}

\date{\today}
\maketitle

\section{Introduction}
The Cuntz semigroup of a not necessarily unital ring $R$, denoted by $\SR(R)$, was introduced and studied in \cite{AntAraBosPerVil23:CuRing}. In the unital case, the main idea behind its definition consists of equipping the class of countably generated projective modules with an equivalence relation, generally weaker than isomorphism.~This semigroup is undoubtedly related to $\V^*(R)$, the monoid of iso\-mor\-phism classes of such modules, and may be thought of as a quotient of the latter that needs to be handled more delicately. A significant difference between these objects is, however, their order structure. Whilst $\V^*(R)$ is algebraically ordered (that is, $x\leq y$ if there is $z$ such that $x+z=y$), this is not the case for $\SR(R)$ except in some particular cases (for example, if $R$ is a unit-regular ring).

The study of countably generated projective modules over a ring has already a long tradition. On the one hand, as the monoid $\V^*(R)$ is an isomorphism invariant, it is interesting to investigate its structure and how it may help distinguish rings in a prescribed class. For example, a complete description of countably generated projective modules over semilocal noetherian rings was given in \cite{HerPri2011} (see also \cite{Herbera2014} and \cite{HerPri2014}); for other classes of rings, such  as von Neumann regular rings, an analysis of their countably generated projective modules was carried out in \cite{AraPerPar2000} (see also the references therein). On the other hand, there is also a close connection between projective modules and the class of pure ideals of the ring via the trace ideal construction, in the sense that any pure right ideal arises in this way and, if the ring is commutative, then the trace of any projective module yields a pure ideal. The situation is different in the noncommutative setting (see, e.g.  \cite{JonTro74}). A much deeper study of idempotent ideals arising as traces of countably generated projective modules was carried out in \cite{Herbera2014}, where a characterization of when an ideal is the trace ideal of a countably generated projective module was given. (See also \autoref{sec:pure}.)

The definition of $\SR(R)$ given in \cite{AntAraBosPerVil23:CuRing} was partly inspired by the construction carried out in \cite{Coward2008} for the class of countably generated Hilbert mo\-du\-les over a C*-algebra $A$ (that is, a norm-closed, self-adjoint subalgebra of the algebra of bounded linear operators on a complex Hilbert space), which yields the Cuntz semigroup $\Cu(A)$.~This semigroup encodes a great deal of information of the algebra. For example, as proved in  \cite{Ciuperca2010}, all closed two-sided ideals and quotients are witnessed by $\Cu(A)$ in the sense that, for any such ideal $I$ of $A$, one has that $\Cu(I)$ is an ideal of $\Cu(A)$ and $\Cu(A/I)\cong \Cu(A)/\Cu(I)$ (with suitable notions of ideal and quotients for these semigroups).~This is in stark contrast with the situation for the semigroup $\V (A)$ of isomorphism classes of finitely generated projective modules, where the isomorphism $\V(A/I)\cong\V(A)/\V(I)$ holds only for special classes of C*-algebras. In a different direction, it also bears recalling that, for the very large class of so-called \emph{classifiable algebras}, a suitable interpretation of the Cuntz semigroup defines a functor equivalent to the Elliott invariant, and thus it contains the same information as (topological) $\K$-theory and traces; see \cite{AntDadPerSan11}.

The Cuntz semigroup of a C*-algebra belongs to an abstract category of semigroups called $\Cu$, in such a way that the assignment $A\mapsto \Cu(A)$ is continuous (see \cite{Coward2008, APT-Memoirs2018}), a particularly relevant fact since models for the class of classifiable algebras come in the form of an inductive limit decomposition.

Besides the picture of $\SR(R)$  involving countably generated projective modules, we also provide in \cite{AntAraBosPerVil23:CuRing} a description of this semigroup based on (sequences of) elements in arbitrary matrices over the ring. More explicitly, for elements $x,y\in R$, one writes $x\precsim_1 y$ provided $x=rys$, for some $r,s\in R$. Then, the semigroup $\SR(R)$ may be seen as a subsemigroup of a larger object, termed $\Lambda(R)$, which is a quotient of the set of all increasing sequences of elements in matrices over $R$ with respect to the antisymmetrization of $\precsim_1$. In this picture, the semigroup $\SR(R)$ is then built out of the latter using a particular type of sequences very closely related to a description of countably generated projective modules as inductive limits of free modules; see \cite{Whthead80}, \autoref{sec:Prelims}, and also \cite[Proposition 2.13]{AntAraBosPerVil23:CuRing}  for more details.

This paper has two main goals. First, for any ring $R$, we explore the structure of the semigroups $\SR(R)$ and $\Lambda(R)$ and their relation with the ideal lattice of $R$ and also to its quotients.~(See Sections \ref{sec:decomposable}, \ref{sec:pure}, \ref{sec:QuotientsIdeals} and \ref{sec:scu}).~Secondly, we study when $\SR(R)$ and $\Lambda(R)$ are in $\Cu$, and the continuity of $\SR(\text{--})$ and $\Lambda(\text{--})$ as functors, with applications to the functor $\SCu(\text{--})$ introduced in \cite{AntAraBosPerVil23:CuRing}, where $\SCu(R)=(\Lambda(R),\SR(R))$ for a weakly $s$-unital ring $R$. (See Sections \ref{sec:leftdense}, \ref{sec:IndLim} and \ref{sec:ContinuitySCu}.)

When analysing the ideal lattice of $\Lambda(R)$ and $\SR(R)$, we are led to introduce the notions of \emph{decomposable} and  \emph{quasipure} ideals, respectively. In short, an ideal $I$ of a ring $R$ is decomposable if, for any $x\in M_\infty(I)$, there is $y\in M_\infty(I)$ with $x\precsim_1 y$; see \autoref{par:ideals}. This is a notion very much devised for general rings, as any two-sided ideal in a unital (or even weakly $s$-unital) ring is automatically decomposable. Quasipure ideals, on the other hand, are decomposable ideals where the comparison relation above satisfies an additional requirement; see \autoref{def:QuaPurId}. As it turns out, an ideal of a unital ring is quasipure precisely when it is the trace ideal of a projective right module; see \autoref{lma:CarProj}. We show that the lattice $\mathrm{Lat}(\Lambda(R))$ of order-ideals of $\Lambda(R)$ captures all decomposable ideals of $R$, and that the lattice of order-ideals $\mathrm{Lat}(\SR(R))$ of the smaller semigroup $\SR(R)$ is still large enough to witness the two-sided \emph{quasipure} ideals of $R$. More precisely, denoting by $\mathrm{Lat}_\mathrm{d}(R)$ the lattice of decomposable ideals and by $\mathrm{Lat}_\mathrm{qp}(R)$ the lattice of quasipure ideals, we prove:
\begin{thmintro}[cf. \ref{Bij_SR}, \ref{Bij_WR}, \ref{thm:retract} and \ref{thm:quo}]
	Let $R$ be any ring. Then
	\begin{enumerate}[\rm (i)]
		\item There are lattice isomorphisms 
		\[
		\mathrm{Lat}_\mathrm{d}(R)\cong \mathrm{Lat}(\Lambda(R)) \text{ and } \mathrm{Lat}_\mathrm{qp}(R)\cong \mathrm{Lat}(\SR(R)).
		\]
		\item The lattice $\mathrm{Lat}_\mathrm{qp}(R)$ is a complete sublattice (and, as a partially ordered set, a retract of $\mathrm{Lat}_\mathrm{d}(R)$).
		\item Given a decomposable ideal $I$ of $R$, we have $\Lambda(R)/\Lambda_R(I)\cong\Lambda(R/I)$. If, furthermore, $I$ is quasipure, then $\SR(R)/\SR(I)\cong \SR(R/I)$.
	\end{enumerate}
\end{thmintro}	
As the structures of these two semigroups are intimately related we consider, for any ring $R$, the pair $\SQ (R):=(\Lambda(R), \SR(R))$ as a more general version of the pair $\SCu(R)$ defined in  \cite{AntAraBosPerVil23:CuRing} for weakly $s$-unital rings. 
We also define an abstract category $\SQ$ in which each pair $\SQ (R)$ lies. Furthermore, it turns out that  the assignment $\mathrm{Rings}\to \SQ$ is functorial and the study of this functor is key in order to understand the invariant $\SQ$. 
In the language of \emph{ideals} in $\SQ (R)$ (studied and developed in \autoref{sec:scu}) we show:

\begin{thmintro}[\ref{thm:LatIsoIdeals}]
	Let $R$ be any ring. Then, the map
	\[
	\xymatrixrowsep{0pc}
	\xymatrix{
		{\rm Lat}_{\rm{d}}(R) \ar[r] & {\rm Lat}\left(\SQ(R))\right) \\
		I \ar@{|->}[r] & (\Lambda_R (I),\SR (I))
	}
	\]
	is a lattice isomorphism, and $I$~is quasipure if and only if~$\SR (I)$~is cofinal in $\Lambda_R (I)$.
\end{thmintro}

It was shown in \cite[Proposition 2.13]{AntAraBosPerVil23:CuRing} that $\Lambda(R)$ is an object in the category $\Cu$ alluded to above whenever $R$ is a weakly $s$-unital ring.~However, the question of whether this remained true for more general classes of rings or even whether $\SR(R)$ was an object of $\Cu$ in general was left unanswered.~We partly clarify the situation here by introducing the classes of \emph{dense} rings and \emph{left normal} rings.~Loosely speaking, the first ones are those for which the relation $\precsim_1$ is dense, whilst the second ones are modelled after the condition of normality for topological spaces. Notably, any idempotent ring is dense (hence also any weakly $s$-unital ring).~The class of left normal rings is also pleasantly large, including all (weakly) semihereditary rings (hence all von Neumann re\-gu\-lar rings), all SAW*-algebras closed under the passage to matrix rings, and all ultramatricial algebras. As a byproduct, these definitions turn out to be sufficient to show continuity. We prove:

\begin{thmintro}[cf. \ref{thm:dense} and \ref{thm:limitlambda}] 
	\label{thmC}
	Let $R$ be any ring.
	\begin{enumerate}[\rm (i)]
		\item If $R$ is dense, then $\Lambda(R)$ is an object of $\Cu$, and the assignment $R\mapsto \Lambda(R)$ is continuous when restricted to the class of dense rings.
		\item If $R$ is left normal, then $\SR(R)$ is an object of $\Cu$, and the assignment $R\mapsto \SR(R)$ is continuous when restricted to the class of left normal rings.
	\end{enumerate}
\end{thmintro}	

The functor $\SCu(\text{--})$ (or the more general version $\SQ(\text{--})$) is not continuous in general; see \autoref{exa:SCunotcont}. However, as proved in \autoref{thm:DirLimEx}, the abstract category $\SCu$ admits general inductive limits, and we have:

\begin{thmintro}[\ref{thm:continuity-in-normalcase}]
	\label{thmD}
	Let $((R_{\lambda})_{\lambda\in\Omega},(\phi_{\mu,\lambda})_{\mu\geq\lambda})$ be a direct system of dense, left normal rings. Then $\lim \SCu(R_\lambda) \cong \SCu (\lim R_\lambda)$.
\end{thmintro}

The class of weakly semihereditary rings is closed under direct limits, as observed in \cite{BerDick78}. In particular we obtain from Theorems \ref{thmC} and \ref{thmD} that the assignments $R\mapsto \SR(R)$ and $R\mapsto \SCu(R)$ define continuous functors from the category of (unital) weakly semihereditary rings to the categories $\Cu$ and $\SCu$, respectively.

\section{Preliminaries}
\label{sec:Prelims}

In this section we recall the basic notions that will be needed throughout the paper, many of them already discussed or introduced in \cite{AntAraBosPerVil23:CuRing}.

Given a ring $R$, we denote by $M_\infty (R)$ the ring of infinite matrices with only a finite number of nonzero entries. Given $x=(x_{i,j})_{i,j}\in M_\infty (R)$, we say that $y\in M_n (R)$ is a \emph{finite representative} of $x$ if $y=(x_{i,j})_{i,j\leq n}$ and $x_{i,j}=0$ whenever $i>n$ or $j>n$; see \cite[Section~2]{AntAraBosPerVil23:CuRing}. There are three semigroups that play an important role in the theory of Cuntz semigroups for rings. We define them below:
\begin{parag}[The semigroups $\W(R)$, $\SR(R)$, and $\Lambda(R)$]
	\label{par:WRetal}
	Let $\POM$ denote the category of positively ordered monoids. Morphisms in this category are those monoid maps that preserve addition, order, and the zero element. We denote by $\POM(M,N)$ the set of $\POM$-morphisms between $M$ and $N$. Recall that a monoid is a semigroup with a neutral element.
	
	Let $R$ be a ring. Recall from \cite[Paragraph~2.4]{AntAraBosPerVil23:CuRing} that $R$ is said to be \emph{weakly $s$-unital} if for every $n\geq 1$ and $x\in M_n (R)$ there exist elements $s,t\in M_n (R)$ such that $x=sxt$.
	
	Given two elements $x,y$ in any ring $R$, we write $x\precsim_1 y$ whenever $x=syt$ for some $s,t\in R$. Note that, if $R$ is weakly $s$-unital, then $x\precsim_1 x$ for every $x\in M_\infty (R)$. We also write $x\sim_1 y$ provided $x\precsim_1 y$ and $y\precsim_1x$. 
	
	Set $\W(R)=M_{\infty}(R)/{\sim_1}$, and denote by $[x]$ the class of $x\in M_\infty(R)$ with respect to the relation $\sim_1$. It is proved in \cite[Lemma 2.6]{AntAraBosPerVil23:CuRing} that, if $R$ is weakly $s$-unital, then $\W(R)$ is a positively ordered abelian semigroup with order induced by $\precsim_1$, addition given by $[x]+[y]=[x\oplus y]$, and neutral element $[0]$.	Here, $x\oplus y$ is the infinite matrix represented by the rectangular matrix $\left(\begin{smallmatrix} x & 0 \\ 0 & y\end{smallmatrix}\right)$ as in the comments previous to this paragraph.
	
	Now, for any ring $R$, denote by $T (R)$ and $\preS (R)$ the following sets
	\[
	\begin{split}
		T(R) &=\{(x_n)\mid x_n\in M_\infty(R) \text{ and } x_n\precsim_1 x_{n+1}\text{ for all }n \}
		,\quad\text{and}\quad \\
		\preS (R) & = \{(x_n)\in T(R)\mid x_n=y_{n+1}  x_{n+1}x_{n} \text{ for some }y_{n+1} \text{ for all }n \}.
	\end{split}
	\]
	Note that $\preS (R)\subseteq T(R)$.
	
	Given $(x_n),(y_n)\in T(R)$, we write $(x_n)\precsim (y_n)$ if for every $n$ there exists $m$ such that $x_n\precsim_1 y_m$. We say that $(x_n)$ is \emph{equivalent} to $(y_n)$, in symbols $(x_n)\sim (y_n)$, if $(x_n)\precsim (y_n)$ and $(y_n)\precsim (x_n)$. 
	
	We define
	\[
	\Lambda (R):= T(R)/{\sim},\quad\text{and}\quad 
	\SR (R):= \preS (R)/{\sim},
	\]
	and view $\SR(R)\subseteq\Lambda(R)$.
	
	It is proved in \cite[Lemma 2.6, Paragraph 4.1]{AntAraBosPerVil23:CuRing} that $\SR (R)$ and $\Lambda(R)$ are  positively ordered semigroups, when equipped with the order induced by $\precsim$ and the addition induced by the componentwise diagonal sum, that is, $[(x_n)]+[(y_n)]=[(x_n\oplus y_n)]$.
	
	If the ring $R$ is weakly $s$-unital, the semigroup $\W(R)$ determines  $\Lambda(R)$, in the sense that $\Lambda(R)\cong \Lambda_{\sigma}(\W(R))$, the semigroup of countably generated intervals in $\W(R)$ (see \cite[Proposition 2.17]{AntAraBosPerVil23:CuRing}, and also \autoref{par:algebraic}). 
\end{parag}

The following notions will play an important role in the sequel.
\begin{parag}[Auxiliary relations]
	\label{par:auxi}
	Let $(P,\leq)$ be a partially ordered set. An \emph{auxiliary relation} on $P$ is a binary relation $\prec$ stronger than $\leq$ (i.e. $x\prec y \implies x\le y$ for $x,y\in P$) such that, for any $x',x,y,y'\in P$ with $x'\leq x\prec y\leq y'$, one has $x'\prec y'$. If, further, $P$ is also a monoid, the auxiliary relation is termed \emph{additive} if $0\prec x$ for any $x\in P$ and if, whenever $x_1,x_2,y_1,y_2\in P$ satisfy $x_1\prec y_1$ and $x_2\prec y_2$, we have $x_1+x_2\prec y_1+ y_2$.  
\end{parag}
\begin{parag}[The categories $\Cu$ and $\Q$]
	\label{par:QCu}
	Given a partially ordered set $P$ where suprema of increasing sequences exist, we write $x\ll y$ whenever for every increasing sequence $(z_n)$ in $P$ such that $y\leq \sup_n z_n$, there exists $m$ such that $x\leq z_m$; see \cite{GieHof+03Domains}. This is an example of an auxiliary relation as defined above.
	
	As introduced in \cite{Coward2008}, a positively ordered monoid $S$ is an \emph{abstract $\Cu$-semigroup} if it satisfies the following four conditions:
	\begin{enumerate}
		\item[(O1)] Every increasing sequence $(x_n)$ in $S$ has a supremum $\sup _n x_n\in S$.
		\item[(O2)] Every element $x\in S$ is the supremum of a sequence $(x_n)$ such that $x_n \ll x_{n+1}$ for all $n$. We say that $(x_n)$ is a {\it rapidly increasing sequence}.
		\item[(O3)] If $x',x,y',y\in S$ satisfy $x'\ll x$ and $y'\ll y$ then $x'+y'\ll x+y$.
		\item[(O4)] If $(x_n)$ and $(y_n)$ are increasing sequences in $S$, then $\sup _n (x_n+y_n)= \sup_n x_n + \sup_ny_n$.  
	\end{enumerate}
	The morphisms in this category, termed \emph{$\Cu$-morphisms}, are those semigroup maps that preserve all the structure, that is, addition, the zero element, order, the relation $\ll$,  and suprema of increasing sequences. We denote by $\Cu(S,T)$ the set of $\Cu$-morphisms between $S$ and $T$.
	
	The category $\Cu$ was introduced to establish a natural abstract framework to study the Cuntz semigroup of C*-algebras. In fact, it was shown in \cite{Coward2008} that, for any C*-algebra $A$, the so-called Cuntz semigroup $\Cu(A)$ of $A$ is an object of $\Cu$. This category has since then been analysed extensively; see \cite{APT2011, APT-Memoirs2018, GarPer2022} among others. 
	
	The category $\Q$ was introduced in \cite[Definition 4.1]{APT-IMRN2020}, and has as objects those positively ordered semigroups $S$ equipped with an additive auxiliary relation $\prec$ that satisfy axioms (O1) and (O4). We call these objects $\Q$-semigroups. A morphism between $\Q$-semigroups is a monoid morphism that preserves the auxiliary relation and suprema of increasing sequences. We refer to such morphisms as $\Q$-morphisms. One naturally sees the category $\Cu$ as a full subcategory of $\Q$ (with $\prec = \ll$).
	\begin{remark} 
		\label{rmk:prec}
		Given any ring $R$, the semigroups $\Lambda(R)$ and $\SR(R)$ are objects in the category $\Q$, with the auxiliary relation defined as follows: Given $[(x_n)], [(y_m)]$ in either $\Lambda(R)$ or $\SR(R)$, we write $[(x_n)]\prec [(y_m)]$ provided there is $m$ such that $x_n\precsim_1 y_m$ for all $n$. 
		
		Further, using the construction of suprema in both $\Lambda(R)$ and $\SR(R)$ (see \cite[Proposition 2.13, Lemma 4.3]{AntAraBosPerVil23:CuRing}), it is easy to verify that the relation $\prec$ just defined is in general stronger than the compact containment relation.
		
		Note that, in case $R$ is weakly $s$-unital,  $\prec$ agrees with $\ll$ for the semigroup $\Lambda(R)$ and, in fact, $\Lambda(R)$ is an abstract Cuntz semigroup. This fact will be subsumed in \autoref{sec:IndLim}. However, in general it is unclear whether the relations $\prec$ and $\ll$ coalesce, and whether $\Lambda(R)$ or $\SR(R)$ are $\Cu$-semigroups. The more exact relationship between $\Lambda(R)$ and $\SR(R)$ will be explored in \autoref{sec:scu}.
	\end{remark}
\end{parag}
\section{Decomposable ideals}
\label{sec:decomposable}
In this section we introduce the notion of \emph{decomposable ideal} in an arbitrary ring; see \autoref{par:ideals}. In the case of unital or weakly $s$-unital rings, all ideals are decomposable. We show in \autoref{Bij_SR} that decomposable ideals form a lattice, isomorphic to the lattice of ideals of the semigroup $\Lambda(R)$.

\begin{parag}[Ideals in semigroups]
	\label{par:idealssemigroups}
	Let $(P,\leq)$ be a partially ordered set and $X\subseteq P$. Recall that $X$ is \emph{downward hereditary} if, whenever $x\leq y$ in $P$ with $y\in X$, one has $x\in X$.
	
	If $S$ is a $\Q$-semigroup, we say that an \emph{ideal} of $S$ is a downward hereditary subsemigroup $I$ which is closed under suprema of increasing sequences. This  is in line with the already existing notion of ideal for a $\Cu$-semigroup; see \cite{Ciuperca2010} and also \cite[Section 5.1]{APT-Memoirs2018}.
	
	The set of ideals of a $\Q$-semigroup $S$ forms a lattice, which we denote by $\mathrm{Lat}(S)$. Note that, for ideals $I,J$ in $S$, we have $I\wedge J=I\cap J$, whereas $I\vee J=\cap\{K\in \mathrm{Lat}(S)\mid K\supset I,J\}$. In the case that $S$ is, furthermore, a $\Cu$-semigroup, then one may describe $I\vee J$ as
	\[
	I\vee J=\{\sup a_n\mid a_n\ll a_{n+1}\text{ for all }n,\text{ and } a_n\leq y_n+z_n, y_n\in I, z_n\in J\},	
	\]
	as shown in \cite[Paragraph 5.1.6]{APT-Memoirs2018}.
\end{parag}

We now study the relationship between the ideals of any ring $R$ and the ideals of the $\mathcal Q$-semigroup $\Lambda(R)$. 

\begin{lemma}
	\label{lma:ideal}
	Let $R$ be a ring and let $I$ be a two-sided ideal of $R$. Then
	\[
	\Lambda_R(I):=\{ [(x_{n})]\in\Lambda (R)\mid x_{n}\in M_\infty (I)\,\, \forall n\geq 1 \}
	\] is an ideal in $\Lambda(R)$.
	
	Moreover,  if $I$ is a weakly $s$-unital ring, then we may identify $\Lambda_R(I)$ with~$\Lambda(I)$.
\end{lemma}
\begin{proof} Take $(x_{n}),(y_{n})\in T (R)$ such that $(x_{n})\precsim(y_{n})$. Thus, for each $n$, there is $m$ such that $x_n\precsim_1 y_m$, and this implies that $x_{n}\in M_\infty (I)$ for every $n$ such that $y_{m}\in M_\infty (I)$. Therefore, the set
	\[
	\Lambda_R(I):=\{ [(x_{n})]\in\Lambda (R)\mid x_{n}\in M_\infty (I)\,\, \forall n\geq 1 \}
	\]
	is downward hereditary. Observe that $\Lambda_R(I)$ is also a submonoid of $\Lambda(R)$. Furthermore, taking into account (the proof) that $\Lambda(R)$ satisfies (O1) (see, \cite[Proposition 2.13]{AntAraBosPerVil23:CuRing}), we see that $\Lambda_R(I)$ is also closed under suprema of increasing sequences, and therefore it is an ideal of $\Lambda (R)$.
	
	Suppose now that  $I$ is a weakly $s$-unital ring. Notice that, if $x,y\in M_\infty(I)$, then $x\precsim_1 y$ in $M_\infty(R)$ if and only if $x\precsim_1 y$ in $M_\infty(I)$. Indeed, if $x=syt$, for $s,t\in M_\infty(R)$, then using that $I$ is weakly $s$-unital we find $a,b\in M_\infty(I)$ such that $y=ayb$ and thus $x=(sa)y(bt)$ with $sa,bt\in M_\infty(I)$. Therefore $\Lambda_R(I)$ can be identified with $\Lambda(I)$.	
\end{proof}

Let $R$ be a ring. Denote by $\mathrm{Lat}(R)$ the lattice of two-sided ideals of $R$ and by $\mathrm{Lat}(\Lambda(R))$ the lattice of ideals of $\Lambda(R)$. By \autoref{lma:ideal}, we may define
\[
\xymatrixrowsep{0pc}
\xymatrix{
	{\rm Lat}(R) \ar[r]^-{\psi_{\Lambda}} & {\rm Lat}(\Lambda (R))\\
	I \ar@{|->}[r] & \Lambda_R(I)
}
\]
which is an ordered morphism. We will now define the class of ideals needed so that $\psi_\Lambda$ is a lattice isomorphism.	
\begin{parag}[Decomposable ideals]
	\label{par:ideals} 
	Let $R$ be a ring. We say that an ideal $I$ of $R$ is \emph{decomposable} if, for any $x\in M_\infty(I)$, there is $y\in M_\infty(I)$ such that $x\precsim_1 y$ in $M_{\infty }(R)$. This is a notion very much intended for non-unital rings, in the sense that if $R$ is weakly $s$-unital, then any ideal is already decomposable. Any ideal that is weakly $s$-unital as a ring is also decomposable.
	
	We use $\mathrm{Lat}_\mathrm{d}(R)$ to denote the subset of $\mathrm{Lat}(R)$ consisting of the decomposable ideals of $R$. Notice that $\mathrm{Lat}_\mathrm{d}(R)$ is also a lattice. To see this, let $I,J$ be decomposable ideals, and let $x\in I, y\in J$. Then there are $\tilde{x}\in M_\infty(I)$, $\tilde{y}\in M_\infty(J)$ such that $x\precsim_1 \tilde{x}$ and $y\precsim_1\tilde{y}$. Using that 
	$x=r\tilde{x}r'$ and $y=s\tilde{y}s'$ for some $r,r',s,s'\in M_\infty (R)$, we get
	\[
	x+y\!=\!
	\left(\begin{array}{cc}x+y&0\\0&0 \end{array}\right)
	= \left(\begin{array}{cc}r&s\\0&0 \end{array}\right)
	\left(\begin{array}{cc}\tilde{x}&0\\0&\tilde{y}\end{array}\right)
	\left(\begin{array}{cc}r'&0\\s'&0 \end{array}\right)
	\precsim_{1}
	\left(\begin{array}{cc}\tilde{x}&0\\0&\tilde{y}\end{array}\right)\in M_\infty(I+J),
	\]
	and therefore the supremum of $I$ and $J$ is $I+J$. The infimum of $I$ and $J$ is the ideal
	\[
	I\wedge J:=\{x\in R\mid x\precsim_1 y\text{ for some }y\in M_{\infty}(I\cap J)\}.
	\]
	Note that $I\wedge J$ is an ideal since for $x,y\in I\wedge J$ we have $x\precsim_1 \tilde{x}, y\precsim_1\tilde{y}$, for $\tilde{x},\tilde{y}\in M_\infty(I\cap J)$, and therefore $x+y\precsim_1\tilde{x}\oplus\tilde{y}$, where the latter belongs to $M_\infty(I\cap J)$. If we also write $x=a\tilde{x}b$ and $r\in R$, we have $rx=(ra)\tilde{x}b$, where $ra$ is the matrix whose entries are the entries of $a$ multiplied by $r$ on the left. Thus $rx\precsim_1 \tilde{x}$, whence $rx\in I\wedge J$. Likewise $xr\in I\wedge J$.
	
	It is now easy to verify that any decomposable ideal $K$ contained in both $I,J$ must already be contained in $I\wedge J$. We have, by construction, that $IJ\subseteq I\wedge J\subseteq I\cap J$.
\end{parag}
\begin{remark}\begin{enumerate}[(i)]
		\item It is easy to check that every idempotent ideal is decomposable.
		\item We also note that every \emph{closed} two-sided ideal $I$ in a C*-algebra $A$ is automatically decomposable. Indeed, given $x\in I$, choose $0<\alpha, \beta$ such that $\alpha+\beta<1/2$. Then, by \cite[Proposition 1.4.5]{Ped79}, there is $u\in A$ such that $x=u(x^*x)^{\alpha+\beta}$. Now, just note that $x=u(x^*x)^\alpha(x^*x)^\beta\precsim_1 (x^*x)^\alpha$ and $(x^*x)^\alpha\in I$ since $I$ is closed.
		\item Non-closed ideals of C*-algebras have raised interest as of late, and they also provide with some examples of decomposable ideals. As shown in \cite{gardella2023semiprimeidealscalgebras}, any semiprime ideal of a C*-algebra is idempotent, and so decomposable by (i). 
	\end{enumerate}
\end{remark}
We will need the following lemma. In what follows, denote by $R^+=\Z\oplus R$ the Dorroh extension of $R$, and view $R$ as a two-sided ideal of $R^+$.

\begin{lemma}\label{lem:decomposable}
	Let $R$ be any ring, and let $a=(a_{ij})\in M_n(R)$.
	\begin{enumerate}[{\rm (i)}]
		\item Suppose that $a\precsim_1 b$, where $b\in M_\infty(R)$. Then, for each $i,j$, we have that $a_{ij}\precsim_1b$. 
		\item Suppose that, for each $i,j$, we have $a_{ij}\precsim_1 b_{ij}$, for some $b_{ij}\in M_\infty(R)$. Then, there is $\tilde{b}\in M_\infty(R)$ such that $a\precsim_1 \tilde{b}$. Moreover, if $I$ is an ideal of $R$ such that $b_{ij}\in M_\infty(I)$, then we can choose $\tilde{b}\in M_\infty(I)$.
	\end{enumerate}
\end{lemma}
\begin{proof}
	(i):  Let us denote by $E_{i,j}$ the elementary matrix with $1\in R^+$ in the $(i,j)$-position and 0 elsewhere. These matrices need not belong to $M_\infty(R)$, but $E_{i,j}x$ and $xE_{i,j}$ are in $M_{\infty}(R)$  for all $x\in M_{\infty}(R)$.
	
	Now if $a\precsim_1 b$, we have $a=rbs$ with $r,s\in M_\infty(R)$. Hence,
	\[ a_{i,j}=\left(\begin{array}{cc} a_{ij} & 0 \\ 0& 0\end{array}\right)=E_{1,i}aE_{j,1}=E_{1,i}(rbs)E_{j,1}=(E_{1,i}r)b(sE_{j,1}).\]
	Seting $r'=(E_{1,i}r)$ and $s'=(sE_{j,1})$ we get $a_{i,j}\precsim_1 b$.
	
	(ii): Assume  that $a,b,c,d\in R$ satisfy $a\precsim_1 \tilde{a}$, $b\precsim_1 \tilde{b}$, $c\precsim_1\tilde{c}$, and $d\precsim_1\tilde{d}$, for $\tilde{a},\tilde{b},\tilde{c},\tilde{d}\in M_\infty(R)$.
	
	Write suitable matrix decompositions $a=x\tilde{a}y$, $b=z\tilde{b}t$, $c=r\tilde{c}s$, and $d=u\tilde{d}v$. A routine matrix multiplication shows that
	\[
	\left(\begin{array}{cc}a&b\\c&d \end{array}\right)
	= \left(\begin{array}{cccc}x&z & 0 & 0\\0&0& r & u \end{array}\right)
	\left(\begin{array}{cccc}\tilde{a}&0 & 0 & 0\\0&\tilde{b}& 0 & 0\\ 0 & 0 & \tilde{c}& 0 \\ 0 & 0 & 0 & \tilde{d}\end{array}\right)
	\left(\begin{array}{cc}y&0\\0&t\\ s& 0 \\ 0 & v \end{array}\right).
	\]
	Now by induction this settles the result for any matrix in $M_{2^k}(R)$, and since any matrix in $M_\infty(R)$ is identified with a matrix of size $2^k$ for some $k$ (adding zeros if necessary), the result holds for all matrices in $M_n(R)$. The last part of the statement is clear.	
\end{proof}		
\begin{lemma}\label{lma:PhiLambdaWellDef}
	Let $R$ be a ring and let $J\subseteq \Lambda (R)$ be an ideal. Then, the set 
	\[
	\mathrm{Idl}(J):=\{ x\in R\mid x=x_1 \text{ for some } [(x_n)]\in J \}
	\]
	is a decomposable  two-sided ideal of $R$.
\end{lemma}
\begin{proof}
	Let $x,y\in \mathrm{Idl}(J)$.~By definition, there exist $[(x_{n})],[(y_{n})]\in J$ such that ~$x_{1}=x$~and $y_{1}=y$.~ Using that $x_{1}\precsim_1x_{2}$ and $y_{1}\precsim_1y_{2}$ we get, as in \autoref{par:ideals}, that $	x+y\precsim_1 x_{2}\oplus y_{2}$.
	
	Thus, we have  $(x+y,x_{2}\oplus y_{2},x_{3}\oplus y_{3},\ldots)\precsim (x_{n}\oplus y_{n})$ in $T (R)$. Since $[(x_n)]+[(y_n)]\in J$ and $J$ is downward hereditary, it follows that $x+y\in \mathrm{Idl}(J)$.
	
	Next, let $x\in \mathrm{Idl}(J)$ and $r\in R$. Then there is $[(x_n)]\in J$ with $x=x_1$, and thus $x\precsim_1 x_2$. It follows from similar arguments as in \autoref{par:ideals} that $rx,xr\precsim_1 x_2$, and thus $rx, xr\in \mathrm{Idl}(J)$. Hence $\mathrm{Idl}(J)$ is a two-sided ideal of $R$.
	
	It remains to show that $\mathrm{Idl}(J)$ is decomposable. Let $x=(x_{ij})\in M_n(\mathrm{Idl}(J))$. For each $i,j$ there is by definition a sequence $\tilde{x}_{ij}^{(n)}\in M_\infty(R)$ such that $x_{ij}=\tilde{x}_{ij}^{(1)}$ and $[(\tilde{x}_{ij}^{(n)})]\in J$. Put $\tilde{x}_{ij}=\tilde{x}_{ij}^{(2)}$. Since $\tilde{x}_{ij}\precsim_1 \tilde{x}_{ij}^{(3)}\precsim_1 \cdots$, we may apply \autoref{lem:decomposable} (i) to conclude that all entries in $\tilde{x}_{ij}$ also belong to $\mathrm{Idl}(J)$, and thus by condition (ii) in \autoref{lem:decomposable}, we have that $x\precsim_1 \tilde{x}$ for some $\tilde{x}\in M_\infty(\mathrm{Idl}(J))$.
\end{proof}

\begin{theorem}\label{Bij_SR}
	Let $R$ be any ring. Then, ${\rm Lat}_{\mathrm{d}}(R)\cong {\rm Lat}(\Lambda (R))$ as lattices.
\end{theorem}
\begin{proof}
	Let $\phi_\Lambda$ denote the map
	\[
	\xymatrixrowsep{0pc}
	\xymatrix{
		{\rm Lat}(\Lambda (R)) \ar[r]^-{\phi_{\Lambda}} & {\rm Lat}_\mathrm{d}(R)\\
		J \ar@{|->}[r] & \mathrm{Idl}(J)
	}
	\]
	which is well-defined by \autoref{lma:PhiLambdaWellDef}.~It is trivial that $\phi_\Lambda$ is inclusion-preserving, and so $\phi_\Lambda$ is an ordered morphism.	
	
	Since any ordered isomorphism between ordered lattices is a lattice isomorphism, it suffices to show that  the map $\psi_\Lambda$ defined prior to \autoref{par:ideals} is, when restricted to decomposable ideals, the inverse for $\phi_\Lambda$. That is, we have to show that $\mathrm{Idl}(\Lambda_R(I))=I$ for any decomposable ideal $I$ of $R$ and $\Lambda_R(\mathrm{Idl}(J))=J$ for any ideal $J$ of $\Lambda(R)$.
	
	Now, given a two-sided ideal $I\subseteq R$ (decomposable or not), it is easy to check that $\mathrm{Idl}(\Lambda_R(I))\subseteq I$. For the converse inclusion, let $x\in I$. Applying repeatedly that $I$ is decomposable we find a $\precsim_1$-increasing sequence $x_n\in M_\infty(I)$ such that $x=x_1$. By definition, $[(x_n)]\in \Lambda_R(I)$ and $x$ is the first term in the sequence, hence $x\in\mathrm{Idl}(\Lambda_R(I))$.
	
	Let $J$ be an ideal of $\Lambda (R)$, and let $[(x_n)]\in J$. For each $k$, we have $[(x_n)_{n\geq k}] = [(x_n)_{n\geq 1}]$ and thus $[(x_n)_{n\geq k}]\in J$. Since $x_k\precsim_1 x_{k+1}$, it follows from \autoref{lem:decomposable} (i) that all the entries $(x_k)_{ij}$ of $x_k$ satisfy $(x_k)_{ij}\precsim_1 x_{k+1}$, and thus they all belong to $\mathrm{Idl}(J)$. Therefore $[(x_n)_{n\geq 1}]\in\Lambda_R(\mathrm{Idl}(J))$.
	
	For the converse inequality, given an element $[(x_{n})]\in \Lambda_R(\mathrm{Idl}(J))$, choose for each $n$ a sequence $[(y_{n,m})_{m}]\in J$ such that $y_{n,1}=x_{n}$. Note, in particular, that one has
	\[
	x_{n}=y_{n,1}\precsim_{1} y_{1,2}\oplus y_{2,2}\oplus\ldots \oplus y_{n,2}.
	\]
	
	Since $J$ is closed under suprema of increasing sequences we have that 
	\[
	s:=\sup_{n} ([(y_{1,m})_{m}]+\ldots+[(y_{n,m})_{m}])\in J.
	\] 
	
	By construction, we have that $[(x_{n})]\leq s$  and, as $J$ is downward hereditary, it follows that $[(x_{n})]\in J$. This shows that $\Lambda_R(\mathrm{Idl}(J))=J$ for each $J\in {\rm Lat}(\Lambda (R))$, as desired.
\end{proof}
\begin{corollary}
	\label{cor:Bij}
	Let $R$ be a weakly $s$-unital ring. Then $\mathrm{Lat}(R)\cong\mathrm{Lat}(\Lambda(R))$.	
\end{corollary}
\section{Pure and quasipure ideals}
\label{sec:pure}
In this section we focus on describing what part of the ideal structure of a ring $R$ is captured by the semigroup $\SR(R)$. In light of our previous results, one might suspect that the ideals of $\SR (R)$ also distinguish all decomposable ideals of $R$. However, as showcased in \autoref{exa:IdNoBij} below, this is not always the case. The right notion in this case will be that of \emph{quasipure} ideal; see \autoref{def:QuaPurId} and \autoref{Bij_WR}. We also show in \autoref{thm:retract} that the lattice of quasipure ideals is a retract of the lattice of decomposable ideals.

\begin{example}\label{exa:IdNoBij}
	There exists a unital ring $R$ with different ideals $I,J$ such that $\SR(I)=\SR(J)$. 	
\end{example}
\begin{proof}	
	Let $R$ be a commutative principal ideal domain which is not a field. 
	If $I$ is a proper ideal of $R$ and $\SR (I)\ne 0$, then taking a nonzero element $(a_n)$ in $\SR(I)$, we see that 
	$a_i\in \cap_{n\ge 1} I^n$ for all $i$, which implies that $\cap_{n\ge 1} I^n \ne 0$, a contradiction. Hence $\SR (I) = 0 = \SR (0)$ for all proper ideal $I$ of $R$. 
\end{proof}


\begin{parag}[Pure and quasipure ideals]
	\label{def:QuaPurId}
	Recall that a right ideal $I$ of a unital ring $R$ is \emph{pure} as a right submodule of $R$ if for every $y\in I$ there exists $s\in I$ such that $sy=y$. Of course, this notion does not need the unit of the ring and we will use it in the more general context of not necessarily unital rings. By the arguments in, for example, \cite[Lemma 2.2]{AAlgCol2004}, we see that if $I$ is pure then so is $M_n(I)$ for all $n$.
	
	More generally, we say that a right ideal $I$ of a ring $R$ is \emph{quasipure} if for every $x\in M_{\infty}(I)$ there exist $s,y\in M_{\infty}(I)$ with
	\[
	sy=y,\quad\text{and}\quad x\precsim_{1} y
	\]
	in $M_{\infty}(R)$. Note that if $I$ is pure and $x\in M_\infty(I)$, then by choosing $s,t\in M_\infty(I)$ such that $ts=s$ and $x=sx$, we have $x=ssx$, that is, $x\precsim_1 s$. Therefore, $I$ is quasipure.
\end{parag}	
We will be interested in (two-sided) ideals which are quasipure as right ideals.  The following easy lemma contains some equivalent variants in the definition of quasipureness for ideals.
\begin{lemma}
	\label{lem:EqQuasi-Pure}
	Let $I$ be an ideal of a ring $R$. Then the following conditions are equivalent:
	\begin{enumerate}[{\rm (i)}]
		\item $I$ is a quasipure right ideal of $R$.
		\item For each $x\in M_{\infty}(I)$ there exists $r\in M_{\infty} (R^+)$ and $y,s\in M_{\infty}(I)$ such that 
		$x= ry$ and $sy= y$. 
		\item For each $x\in M_{\infty}(I)$ there exists $r,s,y\in M_{\infty} (I)$ such that 
		$x= ry$ and $sy= y$.
	\end{enumerate}
	
	\begin{proof}
		(i)$\implies $ (ii): Suppose that $I$ is a quasipure ideal of $R$ and take any $x\in M_{\infty}(I)$. There are $a,b\in M_{\infty}(R)$ and $y,s\in M_{\infty}(I)$ such that
		$$x= ayb,\qquad y= sy .$$
		Then $x= a(yb)$ and $yb= s(yb)$, showing (ii).
		
		(ii)$\implies $(iii): Let $x\in M_{\infty}(I)$. By (ii) there is $r\in M_{\infty}(R^+)$ and $y,s\in M_{\infty}(I)$ such that $x= ry$ and $y=sy$. 
		Write 
		$ x=ry = (rs)y = r'y$
		where $r':= rs \in M_{\infty}(I)$. 
		
		(iii)$\implies $ (i): Let $r,s,y\in M_{\infty}(I)$ such that $x= ry$ and $y=sy$. Now take $r',s',t'\in M_{\infty}(I)$ such that $r= r's'$ and $s'=t's'$.
		Then $x= ry= r's'y$ so that $x\precsim_1 s'$, as desired.  
	\end{proof}
\end{lemma}
\begin{parag}[Quasipure and trace ideals]
	Let us denote by $\mathrm{Lat}_\mathrm{qp}(R)$ the subset of $\mathrm{Lat}_\mathrm{d}(R)$ consisting of quasipure right ideals that are also two-sided ideals. As in Paragraph~\ref{par:ideals}, $\mathrm{Lat}_\mathrm{qp}(R)$ forms a lattice. Indeed, the supremum is just given by the sum, and the infimum of two quasipure ideals $I,J$ is the quasipure ideal
	\[
	\{ x\in R\mid  x\precsim_1 y \text{ and }sy=y, \text{ for some }s,y\in M_\infty(I\cap J)\},
	\]
	since any element $x$ of any quasipure ideal $T$ contained in $I\cap J$ satisfies $x\precsim_{1}y$ with $sy=y$ and $s,y\in T\subseteq I\cap J$.
	
	Pure ideals have been considered in the literature, among other things, in connection with the notion of trace ideal. 
	For commutative rings, Vasconcelos (\cite[Theorem 3.1]{Vas73}) showed that all pure ideals are generated by idempotents if, and only if, any projective ideal is the direct sum of finitely generated projective ideals. In fact, 	
	in a commutative unital ring an ideal is pure precisely when it is the trace ideal of a projective module (see \cite[Proposition 1.1]{JonTro74} and also \cite[Corollary~2.13]{Herbera2014} for an alternative proof of this result). 
	In the noncommutative setting, J\o ndrup and Trosborg in \cite{JonTro74} proved that a pure ideal is always the trace ideal of a projective right module, but not conversely (\cite[Example 1.2]{JonTro74}). Other examples are also given in \cite[Remark 2.10 (3)]{Herbera2014}.
	
	Trace ideals are usually considered in the unital seting, but it is straightforward to consider the corresponding notion for non-unital rings. 
	In \cite[Paragraph 4.11]{AntAraBosPerVil23:CuRing}, the authors define the semigroup $\CP (R)$ out of equivalence classes of countably generated projective $R$-modules $P$, which by definition are projective 
	$R^+$-modules such that $P=PR$. Moreover, it is shown in \cite[Theorem 4.13]{AntAraBosPerVil23:CuRing} that there is an isomorphism of ordered monoids $\CP (R)\cong \SR (R)$ for any ring $R$. 
	
	The {\it trace ideal} of a projective $R$-module $P$ is defined as the trace ideal of $P$ as an $R^+$-module, namely
	$\text{tr}(P) = \sum f(P)$, where $f$ ranges on all homomorphisms $f\colon P\to R^+$. Note however that since $P=PR$ we have $\text{tr} (P) \subseteq R$.  
	Hence $\text{tr}(P)$ is always an idempotent ideal of $R$ for any projective $R$-module $P$. 
	
	The exact relationship between the trace ideals of projective right modules and two-sided ideals is captured by the notion of quasipureness, as shown below.
	The main ingredient in this characterization is \cite[Proposition 2.6]{Herbera2014}.
\end{parag}

\begin{lemma}\label{lma:CarProj}
	Let $R$ be a unital ring and let $I\subseteq R$ be a two-sided ideal. Then, the following are equivalent:
	\begin{enumerate}[{\rm(i)}]
		\item $I$ is quasipure.
		\item Given any finite subset $X\subseteq I$, there exist finitely generated left ideals $J_{1}\leq J_{2}\leq I$ such that $X\subseteq J_{1}$ and $J_{2}J_{1}=J_{1}$.
		\item $I$ is the trace ideal of a projective right $R$-module.
	\end{enumerate}
\end{lemma}
\begin{proof}
	That (ii) and (iii) are equivalent is proved in \cite[Proposition~2.6]{Herbera2014}. Thus, we only need to show that (i) is equivalent to (ii).
	Let us first show that (i) implies (ii). 
	
	Suppose $X=\{x_1,x_2,\dots,x_n\}$. Let $x=\left(\begin{smallmatrix} x_1 & \\ \vdots & \textrm{\ \large 0} \\ x_n & \end{smallmatrix}\right)\in M_\infty(I)$. By \autoref{lem:EqQuasi-Pure}, there exists $r,y,s\in M_\infty(I)$ such that $x= ry$ and $sy=y$. Hence we obtain	
	\[
	x_i= r_{i,1}y_{1}+\ldots +r_{i,m}y_{m},
	\]
	where $y_1,\dots,y_m\in I$ are the nonzero coefficients of the first column of $y$.
	
	Let $J_{1}$ be the left ideal generated by $y_{1},\ldots ,y_{m}$, and let $J_{2}$ be the left ideal generated by $y_{1},\ldots ,y_{m}$ and all the non-zero entries in $s$. It follows by  construction that we have $X\subseteq J_{1}$, $J_{1}\leq J_{2}$ and $J_{2}J_{1}=J_{1}$.
	
	We now prove that (ii) implies (i). Thus, let $x\in M_{\infty}(I)$ and let $m\geq 1$ be such that the entries $x_{i,j}$ of $x$ are zero whenever $i>m$ or $j>m$. By assumption, there exist finitely generated left ideals $J_{1}\leq J_{2}$ such that $x_{i,j}\in J_{1}$ whenever $i,j\leq m$ and $J_{2}J_{1}=J_{1}$. Let $y_{1},\ldots ,y_{n}$ be the generators of $J_{1}$, and let $r_{1},\ldots ,r_{m}$ be $m\times n$ matrices such that 
	\[
	\left(\begin{array}{c}
		x_{1,j} \\
		\vdots \\
		x_{m,j} \\
	\end{array}\right) 
	=r_{j}
	\left(\begin{array}{c}
		y_{1} \\
		\vdots \\
		y_{n} \\
	\end{array}\right) 
	\]
	for every $j$.
	
	Let $\overline{y}$ denote the $n\times 1$ column vector $(y_i)_i$. Then, the matrices
	\[
	r= 
	\left(\begin{array}{c}
		r_{1}\,\, r_{2}\,\,\ldots\,\, r_{n}
	\end{array}\right),\quad\text{and}\quad
	y=
	\left(\begin{array}{cccccc}
		
		\overline{y}  &   0     & \ldots   & 0 & 0 & \ldots \\
		0    &   \overline{y}  &   & 0 & 0 & \\
		\vdots &  &  \ddots^{n)} & & \vdots & \\
		0 & 0 & & \overline{y} & 0 & \\
		0 & 0 & \ldots & 0 & 0 & \\
		\vdots & & & & & \ddots \\
	\end{array}\right)
	\]
	satisfy $x=ry$.
	
	Further, since $J_{2}\leq I$ satisfies $J_{2}J_{1}=J_{1}$, there exists $s_0\in M_{n,n}(I)$ such that $s_0\overline{y}=\overline{y}$. Letting $s=\text{diag}(s_0, s_0 ,\ldots , s_0)$ one gets $sy=y$, as required.
\end{proof}
\begin{corollary}
	\label{cor:idptqp}
	Let $R$ be a ring. Then every idempotent ideal that is finitely generated as a left ideal is quasipure.	
\end{corollary}	
\begin{proof}
	By \cite[Corollary 2.7]{Whthead80}, such an ideal is the trace of a countably generated projective right $R^+$-module $P$. Observing that $\text{tr}(P)$ is the ideal of $R^+$ generated by the entries of any column-finite idempotent matrix $E$ representing $P$, we conclude that all entries of $E$ belong to $R$, and hence $P=PR$ is a projective $R$-module. Hence the result follows from \autoref{lma:CarProj}. 
\end{proof}	

We will now see that the lattice of two-sided quasipure ideals of a ring $R$ is isomorphic to the ideal lattice of $\SR (R)$. This extends \cite[Theorem 2.1]{FHK}, where it is shown that the lattice of trace ideals of 
finitely generated projective $R$-modules is isomorphic to the lattice of order-ideals of the monoid $V(R)$, for each unital ring $R$. We first need the following two lemmas.
\begin{lemma}\label{Ideal_WR}
	Let $I$ be a two-sided ideal of a ring $R$, and let $(x_n)$ be a sequence of elements in $M_{\infty } (I)$. Then, $(x_{n})\in\preS (R)$ if and only if $(x_{n})\in\preS (I)$. In particular, one has $\SR(I)=\SR (R)\cap\Lambda_R(I)$.
\end{lemma}
\begin{proof}
	If $(x_{n})$ is in $\preS (I)$, then it is trivially in $\preS (R)$.
	
	Conversely, if $(x_{n})\in\preS (R)$ with $x_{n}\in M_\infty (I)$ for every $n$, we know that $x_{n}= y_{n+1}x_{n+1}x_{n}$ with $y_{n+1}$ possibly not in $M_\infty (I)$. However, one has
	\[
	x_{n}= y_{n+1}x_{n+1}x_{n}= (y_{n+1}y_{n+2}x_{n+2})x_{n+1}x_{n}
	\]
	and, since $y_{n+1}y_{n+2}x_{n+2}\in M_\infty (I)$, it follows that $(x_{n})\in\preS (I)$.
\end{proof}

\begin{lemma}\label{QP_Property}
	Let $I$ be a two-sided quasipure ideal of a ring $R$.~Then, for every $x\in M_\infty(I)$, there exists $(x_{n})\in \preS (I)$ such that $x\precsim_{1} x_{1}$.
\end{lemma}
\begin{proof}
	Given $x\in M_\infty(I)$, we know from \autoref{lem:EqQuasi-Pure} that there exist elements $r\in M_\infty(R)$ and $x_{1},s_{1}\in M_\infty(I)$ such that
	\[
	x= rx_{1},\quad \text{and}\quad s_{1}x_{1}=x_{1}.
	\]
	
	Using once again that $I$ is quasipure for $s_{1}$, we find elements $y_{2}\in M_\infty(R)$ and $x_{2},s_{2}\in M_\infty(I)$ such that
	\[
	s_{1}= y_{2}x_{2},\quad \text{and}\quad s_{2}x_{2}=x_{2}.
	\]
	
	In particular, one gets that $y_{2}x_{2}x_{1}=s_{1}x_{1}=x_{1}$. Proceeding by induction, we obtain a sequence $(x_{n})\in\preS (R)$ with $x_{n}\in M_\infty (I)$ for all $n$. It follows from \autoref{Ideal_WR} that $(x_{n})\in\preS (I)$.
\end{proof}


In view of \cite[Proposition 1.4]{Herbera2014} and \cite[Theorem 2.4]{Whthead80}, it is natural to define the ideal of $R$ associated to an ideal $J$ of $\SR (R)$ as the two-sided ideal of $R$ generated by all the entries
of matrices appearing in the representatives of the elements of $J$. This is indeed the procedure that we follow here. For an ideal $J$ of $\SR(R)$ define
\[
\mathrm{Idl}_{\mathrm{\SR(R)}}(J):=\{x\in R\mid x\text{ is an entry of } x_1\in M_\infty(R)\text{ such that } [(x_n)]\in J\}.
\]

Note that since $[(x_n)] = [(x_n)_{n\ge k}]$ for each $k\ge 1$, $\mathrm{Idl}_{\mathrm{\SR(R)}}(J)$ is indeed the set of all entries of matrices in $M_{\infty}(R)$ appearing in some representative $(x_n)$ of 
some element of $[(x_n)]\in J$.

\begin{lemma}
	\label{lem:constructingK}
	Let $R$ be a ring, and let $J\subseteq \SR (R)$ be an ideal. Then the following hold:
	\begin{enumerate}[\rm (i)]
		\item  $\mathrm{Idl}_{\SR(R)}(J)$ is a right ideal of $R$.
		\item The left ideal of $R$ generated by $\mathrm{Idl}_{\mathrm{\SR(R)}}(J)$ is a two-sided ideal of $R$, which is a quasipure left ideal. 
	\end{enumerate}
\end{lemma}
\begin{proof} (i): Set $I:= \mathrm{Idl}_{\mathrm{\SR(R)}}(J)$. 
	We first show that $I$ is additive. To see this, let $x,y\in \mathrm{Idl}_{\SR(R)}(J)$ and choose $[(x_n)], [(y_n)]\in J$ such that $x$ is the $(i,j)$ entry of $x_1$ and $y$ the $(k,l)$ entry of $y_1$, for some $i,j,k,l$, and with $x_1, y_1$ matrices of the same size (after adding zeros if necessary). We observe next that we can assume without loss of generality that $(i,j)=(1,1)= (k,l)$. Let $\sigma,\tau\in M_{\infty}(\Z)^+
	\subseteq M_{\infty}(R^+)^+$ suitable permutation matrices so that $(\sigma x_1 \tau)_{1,1} = (x_1)_{i,j}=x$. Then $(\sigma x_1\tau, x_2\sigma^{-1},x_3,\dots )\in \mathcal S (R)$ and $[(\sigma x_1\tau, x_2\sigma^{-1},x_3,\dots )] = [(x_n)]$. Since a similar operation can be done with $(y_n)$, we have shown our claim.

	Now, write $x_1=z_2x_2x_1$, and $y_1=t_2y_2y_1$, and let $P\in M_\infty(R^+)$ be given by $P=\left(\begin{smallmatrix} 1 & 1\\ 0 & 1\end{smallmatrix}\right)$ (where $1$ denotes the identity matrix of suitable size). Now
	\[
	\left(\begin{array}{cc} x_1+y_1 & 0\\ y_1 & 0\end{array}\right)=P\left(\begin{array}{cc} x_1 & 0\\ y_1 & 0\end{array}\right)=P\left(\begin{array}{cc} z_2 & 0\\0 & t_2\end{array}\right)\left(\begin{array}{cc} x_2 & 0\\ 0 & y_2\end{array}\right)P^{-1}P\left(\begin{array}{cc} x_1 & 0\\ y_1 & 0\end{array}\right),
	\]
	which implies that, setting $\tilde{w}_2=P\left(\begin{smallmatrix} z_2 & 0 \\ 0 & t_2\end{smallmatrix}\right)$, $\tilde{u}_1=\left(\begin{smallmatrix}x_1+y_1 & 0 \\ y_1 & 0\end{smallmatrix}\right)$, and $\tilde{u}_2=\left(\begin{smallmatrix} x_2 & 0 \\ 0 & y_2\end{smallmatrix}\right)P^{-1}$,  we have $\left(\begin{smallmatrix}x_1+y_1 & 0 \\ y_1 & 0\end{smallmatrix}\right)=\tilde{w}_2\tilde{u}_2\left(\begin{smallmatrix}x_1+y_1 & 0 \\ y_1 & 0\end{smallmatrix}\right)$. Also, $\tilde{u}_2=(z_3\oplus t_3)(x_3\oplus y_3)\tilde{u}_2$. This implies that the sequence $(\tilde{u}_1,\tilde{u}_2,x_3\oplus y_3,\dots)$ is equivalent to $(x_n\oplus y_n)$ and therefore $x+y\in I$.
	
	Now we show that $I$ is closed under right multiplication. To this end, let $x\in \mathrm{Idl}_{\SR(R)}(J)$ and $r\in R$. Then there is $[(x_n)]\in J$  such that $x$ is an entry of $x_1$. Note that, since there is a sequence $(y_n)\in M_\infty(R)$ such that $y_{n+1}x_{n+1}x_n=x_n$, we have in particular that $y_2x_2(x_1r)=x_1r$. This implies that in fact $[(x_1r,x_2,x_3,\dots)]=[(x_n)]$, and $xr$ is of course an entry of $x_1r$. Thus $xr\in I$.
	
	(ii): Since $I$ is a right ideal, the left ideal $K:=R^+I$ generated by $I$ is a two-sided ideal of $R$. To see that $K$ is quasipure, let $x\in M_n (K)$, and let $x_{ij}\in K$ denote its entries.
	Suppose that for each $i,j$ there are $b_{ij}, s_{ij}\in M_{\infty}(K)$ such that $a_{ij}\precsim_1 b_{ij}$ in $M_{\infty}(R)$ and $s_{ij}b_{ij}= b_{ij}$. By (the proof of) \autoref{lem:decomposable}(ii) 
	we have that
	$$ x= (x_{ij}) \precsim_1 \oplus _{i,j} b_{ij}$$
	in $M_{\infty}(R)$. Since in addition $ \oplus _{i,j} b_{ij} = (\oplus_{i,j} s_{ij})(\oplus_{i,j} b_{ij})$, we conclude that $x$ satisfies the definition of quasipurity. Hence we may reduce to consider the case where $x\in K$. 
	Since $K$ is the left ideal generated by $I= \mathrm{Idl}_{\mathrm{\SR(R)}}(J)$, we can write $x= \sum_{i=1}^k r_ia_i$, where $r_i\in R^+$ and $a_i\in I$. Again, since $x\precsim _1 \oplus_{i=1}^k a_i$
	in $M_{\infty} (R^+)$, we can assume that $x\in I$. Assuming that $x\in I$, there is $(x_n)$ with $[(x_n)]\in J$ such that $x$ is an entry of $x_1$. Then $x_1\precsim_1 x_2$ in $M_{\infty}(R)$ and $(b_3x_3)x_2= x_2$ for some $b_3\in M_{\infty}(R)$.
	By \autoref{lem:decomposable}(i), $x\precsim _1 x_2$. Moreover $x_2,b_3x_3\in M_{\infty}(K)$ and $(b_3x_3)x_2 = x_2$. Hence the condition of quasipurity is satisfied by $x$, as desired. 
\end{proof}

The construction in \autoref{lem:constructingK} serves as motivation to consider, for any ideal $J$ of $\SR(R)$, the two-sided ideal $\Tr_R (J):=R^+ \mathrm{Idl}_{\mathrm{\SR(R)}}(J)$. We will refer to $\Tr _R(J)$ as the \emph{trace ideal associated to $J$}. Observe that $\Tr _R(J)$ is the trace ideal of some projective $R$-module, by virtue of \autoref{lma:CarProj}.

\begin{theorem}\label{Bij_WR}
Let $R$ be any ring. Then, the lattices ${\rm Lat}_{\rm qp}(R)$ and ${\rm Lat}(\SR (R))$ are isomorphic.
\end{theorem}
\begin{proof}
Similar to the proof of \autoref{Bij_SR}, we define the maps
\[
\xymatrixrowsep{0pc}
\xymatrix{
	{\rm Lat}_{\rm qp}(R) \ar[r]^-{\psi_{S}} & {\rm Lat}(\SR (R)) \\
	I \ar@{|->}[r] & \SR (I)
}
\text{ and } 
\xymatrixrowsep{0pc}
\xymatrix{
	{\rm Lat}(\SR (R)) \ar[r]^-{\phi_{\SR }} & {\rm Lat}_{\rm qp}(R)\\
	J \ar@{|->}[r] & \Tr _R(J).
}
\]
First, let us see that $\phi_{\SR }\psi_{\SR }(I)=I$ whenever $I$ is quasipure, that is, 
\[
\Tr _R (\SR (I))=I.
\]
Since by definition $\mathrm{Idl}_{\SR(R)}(\SR(I))\subseteq I$, the inclusion $(\subseteq)$ is clear. Thus, let $x\in I$. We know from  \autoref{QP_Property} that there exists $x_{1}\in\mathrm{Idl}_{\SR(R)}(\SR(I))$ such that $x\precsim_1 x_{1}$. By definition, this implies that $x\in \Tr _R (S(I))$.

We now prove that $\psi_{\SR }\phi_{\SR }(J)=J$ for any ideal $J\subseteq \SR (R)$, that is, 
\[
S(\Tr_R (J))=J.
\]	
Let $[(x_n)]\in J$. Then all entries of each $x_n$ belong to $\Tr _R(J)$, and using \autoref{Ideal_WR} we conclude that $[(x_n)]\in \SR(\Tr_R(J))$.

Next, let $[(x_{n})]\in \SR(\Tr _R (J))$. By definition, for each $n$, the $(i,j)$ entry of the matrix $x_n\in M_{s_n}(R)$ has the form $(x_n)_{ij}=\sum_{k=1}^{l_{i,j,n}}r_{i,j,n}^{(k)} a_{i,j,n}^{(k)}$, where $r_{i,j,n}^{(k)}\in R^+$ and $a_{i,j,n}^{(k)}\in \mathrm{Idl}_{\SR(R)}(J)$. Thus, each $a_{i,j,n}^{(k)}$ is an entry of a matrix $y_{i,j,n,1}^{(k)}$ which is part of a sequence $[(y_{i,j,n,m}^{(k)})_m]\in J$, for $k=1,\dots,l_{i,j,n}$. Thus, for each $i,j,n,k$ we have, using \autoref{lem:decomposable} (i), that $a_{i,j,n}^{(k)}\precsim_1 y_{i,j,n,2}^{(k)}$ and therefore $(x_n)_{ij}\precsim_1\oplus_ky_{i,j,n,2}^{(k)}$. Now, the argument in \autoref{lem:decomposable} (ii) implies that $x_n\precsim_1\oplus_{i,j=1}^{s_n}\oplus_{k=1}^{l_{i,j,n}} y_{i,j,n,2}^{(k)}\precsim_1\oplus_{r\leq n}\oplus_{i,j=1}^{s_r}\oplus_{k}^{l_{i,j,r}} y_{i,j,r,2}^{(k)}$. Therefore
\[
[(x_n)]\leq \sup_n\sum_{r\leq n}\sum_{i,j=1}^{s_r}\sum_{k=1}^{l_{i,j,r}}[(y_{i,j,r,m}^{(k)})_m]\in J.
\]
Since $J$ is downward hereditary, we have $[(x_n)]\in J$, as desired.
\end{proof}
\begin{remark}
Let $R$ be a unital ring. Let $(x_n)_n\in \preS (R)$, and identify $(x_n)_n$ with a countably generated projective $R$-module $P$. A combination of the isomorphism $\CP(R)\cong\SR(R)$ (\cite[Theorem 4.13]{AntAraBosPerVil23:CuRing}) with \cite[Theorem~2.4]{Whthead80} shows that the map $\phi_S$ defined in the proof of \autoref{Bij_WR} sends the ideal generated by $[P]$, that is, the set $\{ [(y_n)_n]\colon  [(y_n)_n]\leq \sup_k k[(x_m)_m] \}$ to its trace ideal, that is, $\mathrm{tr} (P)= \Tr _R (\langle [(x_n)] \rangle )$.
\end{remark}
\begin{theorem}
\label{thm:retract}	
Let	$R$ be any ring. Then there are order preserving maps
\[
\varphi\colon\mathrm{Lat}_{\mathrm{qp}}(R)\to\mathrm{Lat}_{\mathrm{d}}(R)\text{ and }\psi\colon \mathrm{Lat}_{\mathrm{d}}(R)\to\mathrm{Lat}_{\mathrm{qp}}(R)
\]
such that
\begin{enumerate}[{\rm(i)}]
	\item $\psi\circ\varphi=\mathrm{id}$ and $\varphi\circ\psi\leq\mathrm{id}$.
	\item $\varphi$ preserves suprema.
	\item $\psi$ preserves infima.
\end{enumerate}
In particular, as a partially ordered set, $\mathrm{Lat}_{\mathrm{qp}}(R)$ is a retract of $\mathrm{Lat}_{\mathrm{d}}(R)$.
\end{theorem}	
\begin{proof}
By Theorems \ref{Bij_SR} and \ref{Bij_WR}, it suffices to show the conclusions of the statement replacing  $\mathrm{Lat}_{\mathrm{qp}}(R)$ by $\mathrm{Lat}(\SR(R))$ and $\mathrm{Lat}_{\mathrm{d}}(R)$ by  $\mathrm{Lat}(\Lambda(R))$.

Upon these identifications, define $\varphi\colon\mathrm{Lat}(\SR(R))\to \mathrm{Lat}(\Lambda(R))$ by 
\[
\varphi(J)=\{z\in \Lambda(R)\mid z\leq y\text{ for some }y\in J\},
\]
which is easily verified to be downward hereditary and closed under addition. If $(z_n)$ is an increasing sequence in $\varphi(J)$, then find $y_n\in J$ such that $z_n\leq y_n$ for each $n$. Then the sequence $(w_n)$ given by $w_n=\sum_{i=1}^{n}y_i$ is increasing in $J$, and if $w=\sup w_n$, clearly $z_n\leq w$ for all $n$, whence $\sup z_n\leq w$. This shows that $\varphi(J)$ is an ideal of $\Lambda(R)$. It is clear that $\varphi$ is order-preserving.

Define $\psi\colon\mathrm{Lat}(\Lambda(R))\to\mathrm{Lat}(\SR(R))$ by $\psi(K)=K\cap\SR(R)$. It is clear that this is a downward hereditary submonoid, also  closed under suprema of increasing sequences. Therefore, it is an ideal of $\SR (R)$.	

Let us verify that (i) holds. Once this is shown, (ii) and (iii) follow easily. Thus, let $J$ be an ideal of $\SR(R)$ and let $x\in \SR(R)\cap\{z\in\Lambda(R)\mid z\leq y\text{ for } y\in J\}$. Since $J$ is an ideal and $x\in \SR(R)$, we have $x\in J$, whence $\SR(R)\cap\varphi(J)\subseteq J$. As the other inclusion is trivial, we have $\psi\circ\varphi=\mathrm{id}$.

For the second part of (i), just note that $\varphi(\psi(K))=\{z\in\Lambda(R)\mid z\leq y \text{ for } y\in K\cap\SR(R)\}\subseteq K$, whenever $K$ is an ideal of $\Lambda(R)$.
\end{proof}
\begin{remark}\label{QP_Kernel}
Let $I$ be a decomposable two-sided ideal of a  ring $R$. (Recall that, if $R$ is weakly $s$-unital, then $I$ can be any ideal.)  By \autoref{thm:retract}, in combination with \autoref{Bij_WR}, there is a unique quasipure ideal $J\subseteq I$ such that $\SR(I)=\Lambda_R(I)\cap\SR(R)=\SR(J)$.	(Note that, in the notation of \autoref{thm:retract}, $J=\psi(I)$, and its uniqueness is given by the fact that $\varphi$ is injective.) The ideal $J= \Tr_R(S(I))$ is the largest trace ideal of $R$ contained in $I$. 
\end{remark}
\section{Quotients by decomposable and quasipure ideals}\label{sec:QuotientsIdeals}

In this section we analyse how the semigroup constructions developed in Sections \ref{sec:decomposable} and \ref{sec:pure} behave with respect to quotients by decomposable and quasipure ideals, respectively; see \autoref{thm:quo}.
\begin{parag}[Quotients in ordered semigroups]
	Let $S$ be a $\Q$-semigroup, and let $J$ be an ideal of $S$. For $x,y\in S$, we define
	\[
	\begin{split}
		x\leq_{J} y &:\!\iff x\leq y+z,\, z\in J,\\
		x\sim_{J} y &:\!\iff x\leq_{J} y \text{ and } y\leq_{J} x.
	\end{split}
	\]
	We denote the quotient $S/{\sim_J}$ by $S/J$ and its elements by $x_J$, for $x\in S$, and we equip $S/J$ with the addition and order induced by the addition in $S$ and $\leq_{J}$, respectively. Using the techniques from \cite[Lemma~5.1.2]{APT-Memoirs2018}, one sees that the quotient $S/J$ is a partially ordered monoid that also satisfies axioms (O1) and (O4). In fact, if $S$ is already a $\Cu$-semigroup, then \cite[Lemma~5.1.2]{APT-Memoirs2018} shows that $S/J$ as defined above is also a $\Cu$-semigroup.
	
	We define the relation 
	\[ x_J\prec_J y_J :\!\iff x\leq y'+z \textrm { and } y' \prec y + w \textrm { for some } y'\in S,\ z,w\in J.\] 
	One can routinely check that this is well defined and is an additive auxiliary relation on $S/J$.
	
	The natural quotient map $\pi_J\colon S\to S/J$, given by $\pi_J(x)=x_J$, is then  a semigroup morphism that preserves suprema of increasing sequences and the auxiliary relation.
\end{parag}
\begin{lemma}
	\label{lma:auxi}
	Let $R$ be any ring, and let $I$ be a decomposable two-sided ideal of $R$. Let $\pi\colon R\to R/I$ denote the quotient map (and any of its amplifications to matrices). If, for $x,y\in M_\infty(R)$, we have $\pi(x)\precsim_1\pi(y)$, then there is $z\in M_\infty(I)$ such that $x\precsim_1 y\oplus z$.
\end{lemma}
\begin{proof}
	By assumption, there are $a,b\in M_\infty (R)$ and $z'\in M_\infty(I)$ such that $x=ayb+z'$. Since $I$ is decomposable, there is $z\in M_\infty(I)$ such that $z'\precsim_1 z$. Therefore, we have that $x\precsim_1 y\oplus z$, as desired.
\end{proof}


\begin{theorem}
	\label{thm:quo}
	Let $R$ be any ring, and let $I$ be a decomposable two-sided ideal of $R$. Then
	\begin{enumerate}[{\rm (i)}]
		\item $\Lambda(R)/\Lambda_R(I)\cong\Lambda(R/I)$.
		\item If, furthermore, $I$ is quasipure, then $\SR (R)/\SR (I)\cong \SR (R/I)$.
	\end{enumerate}
\end{theorem}
\begin{proof}
	Throughout the proof, let us denote by $\pi\colon R\to R/I$ the quotient map. 
	
	(i): The map $\pi$ induces a map $\pi_I:=\Lambda(\pi)\colon\Lambda(R)\to \Lambda(R/I)$ by $\pi_I([(x_n)])=[(\pi(x_n))]$, which in turn we use to define 
	\[
	\bar{\pi}_I\colon \Lambda(R)/\Lambda_R(I)\to\Lambda(R/I) \text{ by }\bar{\pi}_I([(x_n)]_{\Lambda_R(I)})=[(\pi(x_n))],
	\]
	which is easily seen to be a well defined semigroup homomorphism.
	
	Let us prove that $\bar{\pi}_I$ is surjective. Let $[(\pi(x_n))]\in \Lambda(R/I)$. Set $z_1=0$. By \autoref{lma:auxi} applied to $\pi(x_1)\precsim_1\pi(x_2)$, there is $z_2\in M_\infty(I)$ such that $x_1\precsim_1 x_2\oplus z_2$. Another application of \autoref{lma:auxi} to $\pi(x_2\oplus z_2)=\pi(x_2)\precsim_1\pi(x_3)$ yields $z_3\in M_\infty(I)$ with $x_2\oplus z_2\precsim_1 x_3\oplus z_3$. Continuing in this way we find $z_n\in M_\infty(I)$ such that $[(x_n\oplus z_n)] \in \Lambda(R)$. Now $\bar{\pi}_I([(x_n\oplus z_n)])=[(\pi(x_n\oplus z_n))]=[(\pi(x_n))]$.
	
	We now prove that $\bar{\pi}_I$ is an order-embedding. Let $[(x_n)], [(y_n)]\in \Lambda(R)$ be such that $\bar{\pi}_I([(x_n)]_{\Lambda_R(I)})=[(\pi(x_n))]\leq [(\pi(y_n))]= \bar{\pi}_I([(y_n)]_{\Lambda_R(I)})$. After removing certain elements from $(y_n)$ if necessary (withouth changing its class), we may assume that $\pi(x_n)\precsim_1\pi(y_n)$ for each $n$. Using \autoref{lma:auxi} and the fact that $I$ is decomposable, choose $z_{n,m}\in M_\infty(I)$ such that $x_n\precsim_1 y_n\oplus z_{n,1}$ and  $z_{n,m}\precsim_1z_{n,m+1}$ for each $n,m$.
	
	We have $x_1\precsim_1 y_1\oplus z_{1,1}$ and also $x_2\precsim_1 y_2\oplus z_{2,1}\precsim_1 y_3\oplus z_{2,2}\oplus z_{1,2}$, and note that $z_{1,1}\precsim_1 z_{2,2}\oplus z_{1,2}$. Similarly, $x_3\precsim_1 y_3\oplus z_{3,1}\precsim_1 y_4\oplus z_{3,3}\oplus z_{2,3}\oplus z_{1,3}$ with $ z_{2,2}\oplus z_{1,2}\precsim_1z_{3,3}\oplus z_{2,3}\oplus z_{1,3}$. Thus, set $w_1=z_{1,1}$, $w_2=z_{1,2}\oplus z_{2,2}$ and in general $w_n=z_{1,n}\oplus z_{2,n}\oplus\cdots \oplus z_{n,n}$. By construction, $w_n\in M_\infty (I)$ and $w_n\precsim_1 w_{n+1}$. Moreover, for each $n\geq 2$ we have that $x_n\precsim_1 y_{n+1}\oplus w_n$. Therefore $[(x_n)]\leq [(y_n)]+[(w_n)]$ in $\Lambda(R)$, whence $[(x_n)]_{\Lambda_R(I)}\leq [(y_n)]_{\Lambda_R(I)}$.
	
	(ii): Let $\SR (\pi)$ be the induced morphism by the quotient map $\pi\colon R\to R/I$. Given $(x_{n})\in\preS (R/I)$, we know that for each $n$ there exist $y_{n}\in M_\infty(R)$ and $z_{n}\in M_\infty(I)$ such that
	\[
	y_{n+1}x_{n+1}x_{n}+z_{n}=x_{n}.
	\]
	
	Further, since $I$ is quasipure, the proof of \autoref{QP_Property} gives that there exist elements $r_{n}\in M_\infty(R)$ and sequences $(s_{m,n})_{m}\in\preS (I)$ such that $z_{n}=r_{n}s_{1,n}$ for each $n$.
	
	Define the matrices
	\[
	S_{n}= 
	\left(\begin{array}{cccc}
		
		s_{n,1}    &   0           &   \ldots  &     0      \\
		0          &   s_{n-1,2}   &           &     \vdots \\
		\vdots     &               &   \ddots  &     0       \\
		0          &   \ldots      &   0       &  s_{2,n-1}
	\end{array}\right)
	\quad \text{and}
	\quad
	X_{n}= 
	\left(\begin{array}{cc}
		
		x_{n}  &   0        \\
		0      &   S_{n}   \\
		s_{1,n}      &   0      
	\end{array}\right).
	\]

	Note that
	\[
	X_{n+1}X_{n}=
	\left(\begin{array}{ccccc}
		
		x_{n+1}x_{n}  &  0 & 0 & \ldots & 0 \\
		0 & s_{n+1,1}s_{n,1}    &   0           &   \ldots  &     0  \\
		0 & 0          &   s_{n,2}s_{n-1,2}   &           &     \vdots  \\
		\vdots & \vdots     &               &   \ddots  &     \vdots        \\
		0 & 0          &   \ldots      &   0       &  s_{3,n-1}s_{2,n-1} \\
		s_{2,n}s_{1,n} & 0 & \ldots & 0 & 0 \\
		s_{1,n+1}x_{n} & 0 & \ldots & 0 & 0 
	\end{array}\right).
	\]
	
	Thus, given $y_{m,n}$ such that $y_{m+1,n}s_{m+1,n}s_{m,n}=s_{m,n}$, the matrix
	\[
	Y_{n+1}=
	\left(\begin{array}{ccccccc}
		
		y_{n+1}  &  0 & 0 & \ldots & 0 & r_{n}y_{2,n} & 0\\
		0 & y_{n+1,1}    &   0           &   \ldots  &     0 & 0 & 0\\
		0 & 0          &   y_{n,2}   &           &     \vdots & \vdots & \vdots\\
		\vdots & \vdots     &               &   \ddots  &     0       & 0 & 0\\
		0 & 0          &   \ldots      &   0       &  y_{3,n-1} & 0 & 0\\
		0 & 0 & \ldots & 0 & 0 & y_{2,n} & 0
	\end{array}\right)
	\]
	satisfies
	\[
	Y_{n+1}X_{n+1}X_{n}=X_{n}
	\]
	for every $n$.
	
	We have $(X_{n})\in\preS (R)$ and $\SR (\pi)[(X_{n})]=[(x_{n})]$ as required.
	
	Let us now prove that $\SR (\pi )$ is an order-embedding, that is, $\SR (\pi )([(x_{n})])\leq \SR (\pi )([(y_{m})])$ if and only if $[(x_{n})]_{\SR(I)}\leq [(y_{m})]_{\SR (I)}$. Since $\SR ( \pi)$ is a morphism, we only need to show that $\SR (\pi )([(x_{n})])\leq \SR (\pi )([(y_{m})])$ implies $[(x_{n})]_{\SR (I)}\leq [(y_{m})]_{\SR (I)}$.
	
	Take $(x_{n}),(y_{m})\in \preS (R)$ such that $[(\pi(x_{n}))]\leq [(\pi(y_{m}))]$. Upon possibly removing certain elements from $(\pi(y_{m}))$, we may assume that there exist $r_{n},s_{n}\in M_\infty(R)$ and $z_{n}\in M_\infty(I)$ such that
	\[
	x_{n} -r_{n}y_{n}s_{n}=z_{n}
	\]
	for every $n$.
	
	Since $I$ is quasipure, \autoref{QP_Property} implies that $z_{n}\precsim_{1} s_{1,n}$, with $(s_{m,n})_{m}\in\preS (I)$ for every $n$. Thus, we get
	\[
	x_{n}=r_ny_ns_n+z_n\precsim_{1} y_{n}\oplus s_{1,n}\precsim_{1} y_{n+1} \oplus s_{2,n},
	\]
	and thus $x_n\precsim_1 y_{n+1}\oplus (\oplus_{r\leq n} s_{2,r})$.
	This implies
	\[
	[(x_{n})]\leq [(y_{n})] + \sup_{n}( [(s_{m,1})]+\ldots +[(s_{m,n})])
	\]
	and, consequently, one gets $[(x_{n})]_{\SR (I)}\leq [(y_{n})]_{\SR (I)}$.
\end{proof}
\begin{remark}
	Let $R$ be a ring, $I$ a two-sided ideal, and denote by $\pi\colon R\to R/I$ the quotient map.  Under the isomorphism $\SR(R)\cong\CP(R)$ established in \cite[Theorem 4.13]{AntAraBosPerVil23:CuRing}, the map $\SR(\pi)\colon \SR(R)\to \SR(R/I)$ may be identified with the map $\CP(R)\to \CP(R/I)$ given by $[P]\mapsto [P/PI]$. 
	
	It was proved in \cite[Theorem 3.1]{Herbera2014} that countably generated projective $R$-modules can always be lifted modulo the trace ideal of a projective module. That is, if $I$ is the trace ideal of a countably generated projective right $R$-module and $P'$ is a countably generated projective right $R/I$-module, then there is a countably generated projective $R$-module $P$ such that $P/PI\cong P'$. Therefore, if $I$ is quasipure, \autoref{lma:CarProj} shows that $I$ is the trace of a countably generated projective module, and thus by our considerations above, this implies that $\SR(\pi)$ is surjective.
	
	A detailed inspection of \cite[Theorem 3.1]{Herbera2014} and its proof shows that the arguments can be adapted to show surjectivity of $\SR (\pi)$ in the general setting, which is part of the argument carried out in (ii) of \autoref{thm:quo}. Indeed our construction in the proof of \autoref{thm:quo}(ii) of a sequence $(X_n)$ lifting $(x_n)$ gives an actual lifting of the countably generated projective module represented by $(x_n)$.
\end{remark}
\begin{corollary}
	Let $R$ be any ring, and let $J\subseteq \Lambda (R)$ be an ideal. Then, there exist two-sided ideals $\tilde{I}\subseteq I$ of $R$ with $\tilde{I}$ quasipure such that 
	\[
	\Lambda (R)/J\cong\Lambda (R/I) \quad\text{and}\quad \SR (R)/(J\cap \SR (R))\cong \SR (R/\tilde{I}).
	\]
\end{corollary}
\begin{proof}
	Given an ideal $J\subseteq \Lambda (R)$, we know from \autoref{Bij_SR} that $J=\Lambda_R (I)$ for some two-sided decomposable ideal $I\subseteq R$. Consider the ideal $\Lambda_R(I)\cap\SR (R)$ of $\SR (R)$. By \autoref{Bij_WR}, there exists a quasipure ideal $\tilde{I}$ such that $\SR (\tilde{I})= \Lambda_R(I)\cap\SR (R)$. In particular, it follows from \autoref{QP_Property} that $\tilde{I}\subseteq I$. The isomorphisms now follow from \autoref{thm:quo}.
\end{proof}
\section{Ideals and quotients in the category $\SQ$}
\label{sec:scu}
In this section we introduce  the category $\SQ$ and study its relationship with the category $\SCu$, already considered in \cite[Section 5]{AntAraBosPerVil23:CuRing}. Whilst $\SCu$ naturally defines a functor with domain the class of weakly $s$-unital rings, with $\SQ$ we remove this assumption and determine a functor whose domain is the class of arbitrary rings.

We also introduce the notions of ideal and quotients in these categories. To this end we will make use of Theorems \ref{Bij_SR} and \ref{Bij_WR} to show that the lattice of two-sided decomposable (respectively, quasipure) ideals of $R$ is encoded in the lattice of ideals of $\SQ(R)$; see  \autoref{thm:LatIsoIdeals}. 
\begin{parag}[Weakly increasing sequences]
	\label{par:wincr}
	Let $S$ be a $\Q$-semigroup. We say that a sequence $(x_n)$ in $S$ is \emph{weakly increasing} if there exists an increasing sequence $(y_m)$ in $S$ such that
	\begin{enumerate}[(i)]
		\item For every $m$ there exists $n(m)$ such that $y_m\leq x_n$ for every $n\geq n(m)$.
		\item $x_n\leq \sup_m y_m$ for every $n$.
	\end{enumerate}
	Of course, increasing sequences are examples of weakly increasing sequences. In particular, constant sequences are examples of weakly increasing sequences. Observe that this definition does not use anything else other than the existence of suprema of increasing sequences (axiom (O1)) in the given semigroup. Note that necessarily, for a weakly increasing sequence $(x_n)$ as above, we have $\sup_n x_n = \sup_m y_m$, hence suprema of weakly increasing sequences always exist. 
	
	The set of weakly increasing sequences forms a monoid under componentwise addition, with suprema being compatible with addition. If $\varphi\colon S\to T$ is a $\Q$-morphism and $(x_n)$ is a weakly increasing sequence, then $(\varphi(x_n))$ is also weakly increasing with $\varphi(\sup_n x_n)=\sup_n\varphi(x_n)$. Indeed, let  	
	$(x_n),(z_n)$ be weakly increasing sequences, and let $x,z$ be their respective suprema, which we have just noticed exist. It is easy to verify that, by definition, $(x_n+z_n)$ is also a weakly increasing sequence with $x+z=\sup_n (x_n+z_n)$.
	
	In the context of $\Cu$-semigroups, weakly increasing sequences were introduced in \cite[Paragraph 5.1]{AntAraBosPerVil23:CuRing}: We say that a sequence $(x_n)$ in a $\Cu$-semigroup $S$ is weakly increasing if, whenever $x\ll x_n$ for some $n$ and $x$, there exists $m_0$ such that $x\ll x_m$ for every $m\geq m_0$. We prove below that these two notions agree.
\end{parag}
The proof of the following lemma is implicit in \cite[Lemma~5.2]{AntAraBosPerVil23:CuRing}. We offer a few details. For a $\Cu$-semigroup, let us temporarily refer to a weakly increasing sequence as just defined above as a $\Cu$-weakly increasing sequence.
\begin{lemma}\label{prp:Lma54}
	Let $S$ be a $\Cu$-semigroup. A sequence $(x_n)$ is $\Cu$-weakly increasing in $S$ if, and only if, $(x_n)$ is weakly increasing viewing $S$ as a $\Q$-semigroup. 
\end{lemma}
\begin{proof}
	We first note that standard arguments in the theory of abstract Cuntz semigroups allow us to replace the order relation $\leq$ by the relation of compact containment $\ll$ in the definition of a weakly increasing sequence, as follows. A sequence $(x_n)$ in $S$ is weakly increasing if, and only if, there exists a  $\ll$-increasing sequence $(y_k)$ in $S$ such that 
	\begin{enumerate}[{\rm(i)'}]
		\item For every $k$ there exists $n_k$ such that $y_k\ll x_n$ whenever $n\geq n_k$.
		\item $x_n\leq \sup_k y_k$ for every $n$.
	\end{enumerate}	  
	Now, the argument in the proof of \cite[Lemma~5.2]{AntAraBosPerVil23:CuRing} shows that, if $(x_n)$ is a $\Cu$-weakly increasing sequence, then after expressing each element $x_n$ as the supremum of a $\ll$-increasing sequence $(x_n^{(m)})$ and a re-indexing process, one finds a strictly increasing sequence $(m_k)$ in $\N$ such that $x_k^{(m_k)}\ll x_{k+1}^{(m_k)}\ll x_{k+1}^{(m_{k+1})}$. With this, set $y_k=x_k^{(m_k)}$ and one checks that $\sup_k y_k=\sup_n x_n$. Also, by construction, for each $k$, we have that $y_k\ll x_k$, and thus since $(x_n)$ is $\Cu$-weakly increasing there is $n_k$ such that (i)' holds. That (ii)' holds is clear.
	
	It is also clear, on the other hand, that a sequence $(x_n)$ satisfying conditions (i)' and (ii)' above is necessarily $\Cu$-weakly increasing.
\end{proof}
The category $\SCu$ was introduced in \cite[Section 5]{AntAraBosPerVil23:CuRing} with the purpose of balancing out that $\SR(R)$ might not be a $\Cu$-semigroup for a weakly $s$-unital ring $R$, providing an ambient semigroup that does belong to $\Cu$. We define here a category that may be useful for general rings, requiring instead to work with $\Q$-semigroups.
\begin{parag}[The categories $\SQ$ and $\SCu$]
	\label{par:scu}
	Adapting the notation in \cite[Paragraph 5.4]{AntAraBosPerVil23:CuRing}, we let $\SQ$ be the category whose objects are pairs $(S ,W)$ with $S$ a $\Q$-semigroup, and $W$ a submonoid of $S$ closed under suprema of weakly increasing sequences (as defined in \autoref{par:wincr}). This means that, if $(x_n)$ is a sequence in $W$ that is weakly increasing in $S$, then $\sup x_n\in W$. An  \emph{$\SQ$-morphism}  between $(S_1,W_1), (S_2,W_2)\in \SQ$ is a $\Q$-morphism $f\colon S_1\to S_2$ such that $f(W_1)\subseteq W_2$. For brevity, we shall denote an $\SQ$-morphism by $f\colon (S_1,W_1)\to (S_2,W_2)$.
	
	The category $\SCu$ is the full subcategory of $\SQ$ consisting of pairs $(S,W)$, where $S$ is a $\Cu$-semigroup, the auxiliary relation $\prec$ coincides with the way-below relation $\ll$ on $S$, and $W$ is a submonoid of $S$ as above, that is, closed under suprema of weakly increasing sequences.
\end{parag}
\begin{theorem}\label{thm:FunctorSQ}
	Let $R$ be any ring. Then:
	\begin{enumerate}[{\rm (i)}]
		\item The pair $\SQ(R)=(\Lambda(R),\SR(R))$ is an object of $\SQ$.
		\item The assignment $R\mapsto \SQ(R)$ defines a functor $\SQ\colon\mathrm{Rings}\to\SQ$.
	\end{enumerate} 
\end{theorem}
\begin{proof}
	(i): Part of the argument is inspired by the argument carried out in \cite[Proposition 5.6 (i)]{AntAraBosPerVil23:CuRing}. We include full details for convenience.
	
	We need to verify that $\SR(R)$ is a submonoid of $\Lambda(R)$ closed under suprema of weakly increasing sequences. To do so, let $([x_n])$ be a weakly increasing sequence in $\Lambda(R)$ such that $[x_n]\in\SR(R)$ for each $n$. Write $x_n=(x^{(n)}_k)_k$, and we know there are elements $y^{(n)}_k$ such that $y_{k+1}^{(n)}x_{k+1}^{(n)}x_k^{(n)}=x_k^{(n)}$ for each $k$ and $n$.
	
	There is by definition a sequence $([z_m])$ in $\Lambda(R)$ satisfying conditions (i) and (ii) in \autoref{par:wincr}. Write $z_m=(z^{(m)}_k)_k$. Using the description of suprema in $\Lambda(R)$ (see the proof of \cite[Proposition 2.13]{AntAraBosPerVil23:CuRing}), and after a reindexing process, we may assume that $\sup_n[x_n]=\sup_m [z_m]=[(z_m^{(m)})]$.
	
	Since $([x_n])$ is by assumption weakly increasing, for $m=1$, there is $n_1$ such that $[z_1]\leq [x_n]$ whenever $n\geq n_1$. Therefore there is $l_{1}$ such that $z_1^{(1)}\precsim_1 x_{l_{1}}^{(n_1)}$. We also have that $[x_n]\leq [(z_m^{(m)})]$ for each $n$. Therefore, for each $k$, $n$, there is $p_{n,k}$ such that $x_k^{(n)}\precsim_1 z_{p_{n,k}}^{(p_{n,k})}$. Therefore we have $p_{1} := p_{n_1,l_{1}+1}$ for which $x_{l_{1}+1}^{(n_1)}\precsim_1 z_{p_{1}}^{(p_{1})}$.

	Now, arguing as above, we find $n_2>n_1$ and $l_{2}>l_{1}$ such that $ z_{p_{1}}^{(p_{1})}\precsim_1 x_{l_{2}}^{(n_2)}$. Thus in particular we obtain
	\[
	z_1^{(1)}\precsim_1 x_{l_{1}}^{(n_1)}\precsim_1 x_{l_{1}+1}^{(n_1)} \precsim_1 z_{p_{1}}^{(p_{1})}\precsim_1 x_{l_{2}}^{(n_2)}.
	\]
	Continuing in this way we find increasing sequences  $n_m$, $l_m$, $p_m$ of positive integers such that the corresponding sequences $(x_{l_{m}}^{(n_m)})$ and $(z_{p_{m}}^{(p_{m})})$ satisfy
	\[
	z_{p_{{m-1}}}^{(p_{{m-1}})}\precsim_1 x_{l_{m}}^{(n_m)}\precsim_1x_{l_{m}+1}^{(n_m)} \precsim_1 z_{p_{m}}^{(p_{m})}\precsim_1 x_{l_{{m+1}}}^{(n_{m+1})}.
	\]
	In particular, it follows that $\sup[x_n]=\sup[x_{l_{m}}^{(n_m)}]$. 
	
	Write $x_{l_{m}+1}^{(n_m)}=c_{m} x_{l_{{m+1}}}^{(n_{m+1})}d_{m}$ for some $c_{m},d_{m}$. Now we have 
	\[
	x_{l_{m}}^{(n_m)}=y_{l_{m}+1}^{(n_m)}x_{l_{m}+1}^{(n_m)}x_{l_{m}}^{(n_m)}=y_{l_{m}+1}^{(n_m)}c_{m} x_{l_{{m+1}}}^{(n_{m+1})}d_{m}x_{l_{m}}^{(n_m)},
	\]
	and therefore
	\[
	x_{l_{m}}^{(n_m)}d_{{m-1}}=(y_{l_{m}+1}^{(n_m)}c_{m}) (x_{l_{{m+1}}}^{(n_{m+1})}d_{m})(x_{l_{m}}^{(n_m)}d_{{m-1}}).
	\]
	This implies that the sequence $[(x_{l_{m}}^{(n_m)}d_{{m-1}})]$ belongs to $\SR(R)$. From the above observations we also see that $x_{l_{m}}^{(n_m)}d_{{m-1}}\precsim_1 x_{l_{{m+1}}}^{(n_{m+1})}\precsim_1 x_{l_{{m+2}}}^{(n_{m+2})}d_{{m+1}}$. Therefore, $\sup[x_n]\in \SR(R)$, as was to be shown.
	
	(ii): If $f\colon R\to R'$ is a ring homomorphism, then $f$ extends to a homomorphism $f\colon M_\infty(R)\to M_\infty(R')$ in a way compatible with $\precsim_1$ and $\oplus$. Thus, if $[(x_n)]$ belongs to $\Lambda(R)$ or $\SR(R)$, respectively, we have that $[(f(x_n))]$ belongs to $\Lambda(R')$ or $\SR(R')$. Thus the map $\Lambda(f)\colon \Lambda(R)\to \Lambda(R')$ given by $[(x_n)]\mapsto [(f(x_n))]$ is well defined and maps $\SR(R)$ to $\SR(R')$. Also, if $[(x_n)]\prec [(y_m)]$, there is by definition $m$ such that $x_n\precsim_1 y_m$ for all $n$, and thus $f(x_n)\precsim_1 f(y_m)$ for all $n$. This implies that $[(f(x_n))]\prec [(f(y_m))] $. 
	
	Finally, let $[x_n]$ be an increasing sequence in $\Lambda(R)$. Inspection of the proof of \cite[Proposition 2.13]{AntAraBosPerVil23:CuRing} on how the supremum of $[x_n] $ is constructed shows that $\sup[f(x_n)]=\Lambda(f)(\sup[x_n])$.
\end{proof}
\begin{parag}[Ideals in $\SQ$]\label{IdealsSQ}
	Given an object $(S,W)$ in $\SQ$, an \emph{ideal} of $(S,W)$ will be by definition a pair of the form $(I,I\cap W)$, where $I$ is an ideal of $S$ as a $\Q$-semigroup; see \autoref{par:idealssemigroups}. Analogously, one defines the concept of ideal for an object in $\SCu$. We show in \autoref{lma:idealinscu} that any ideal of a pair $(S,W)$ in $\SQ$ (respectively, in $\SCu$) is again an object in $\SQ$ (respectively, in $\SCu$).
	
	The ideals of an object $(S,W)\in\SQ$ form a lattice with the partial order given by inclusion of both components. Indeed, given two ideals $(I,I\cap W)$ and $(J,J\cap W)$,
	their infimum is $(I\cap J,I\cap J\cap W)$, which is clearly an ideal. Further, the supremum of $(I,I\cap W)$ and $(J,J\cap W)$ is $(I\vee J,(I\vee J)\cap W)$, where $I\vee J$ is the supremum of two ideals in $\Q$. Similarly, the ideals of an object $(S,W)\in \SCu$ form a lattice. (See \autoref{par:idealssemigroups}.)
	
	Given a ring $R$, we denote by $\mathrm{Lat}(\SQ(R))$ the lattice of ideals of $\SQ(R)$. Notice that, in case $R$ is weakly $s$-unital, we have that $\SQ(R)=\SCu(R)$ and then $\mathrm{Lat}(\SQ(R))=\mathrm{Lat}(\SCu(R))$.
\end{parag}
\begin{lemma}
	\label{lma:idealinscu}
	Let $(S,W)$ be an object in $\SQ$ and let $(I,I\cap W)$ be an ideal of $(S,W)$. Then, $(I,I\cap W)$ is also an object in $\SQ$. 
\end{lemma}
\begin{proof}
	Note that, by definition, $(I,I\cap W)$ is a pair such that $I$ is an ideal of $S$ and thus in particular $I\in \Q$. Also, $I\cap W\subseteq I$ is a submonoid closed under suprema of increasing sequences. Thus, we only need to check that $I\cap W\subseteq I$ is closed under suprema of weakly increasing sequences.
	
	Let $(x_n)$ be a weakly increasing sequence in $I$ with elements in $I\cap W$. Since $I\in \Q$, we have that the supremum of $(x_n)$ belongs to $I$. We have to check that it also belongs to $W$, and to do so we observe that $(x_n)$ is also weakly increasing as a sequence in $S$ and apply that $(S,W)\in \SQ$. Indeed, since $(x_n)$ is weakly increasing in $I$, there exists an increasing sequence $(y_m)$ in $I$ such that for every $m$ there is $n_m$ with $y_m\leq x_n$ whenever $n\geq n_m$, and such that $x_n\leq \sup_m y_m$ for every $n$. By considering the same sequence $(y_m)$ in $S$, we see that $\sup_my_m$ also belongs to $S$ and satisfies the conditions for $(x_n)$ to be weakly increasing in $S$.
	
	The case where $(S,W)\in \SCu$ is similar.	
\end{proof}

Recall that a subset $X$ of a partially ordered set $P$ is said to be \emph{cofinal} if for every $p\in P$ there exists $x\in X$ such that $p\leq x$.
\begin{theorem}\label{thm:LatIsoIdeals}
	Let $R$ be any ring. Then, the map
	\[
	\xymatrixrowsep{0pc}
	\xymatrix{
		\mathrm{Lat}_{\mathrm{d}}(R) \ar[r]^-{\psi} & {\rm Lat}(\SQ (R)) \\
		I \ar@{|->}[r] & (\Lambda_R(I),\SR (I))
	}
	\]
	is a lattice isomorphism. Further, $I$ is quasipure if and only if $I$ is decomposable and $\SR (I)$ is cofinal in $\Lambda_R(I)$.
	
	If moreover $R$ is weakly $s$-unital, the same map defines a lattice isomorphism $\mathrm{Lat}(R)\cong \mathrm{Lat}(\SCu(R))$.	
\end{theorem}
\begin{proof}
	First note that, since $\SR(I)=\Lambda_R(I)\cap\SR(R)$ by \autoref{Ideal_WR}, we obtain that $\psi$ is well defined and respects inclusion. Also, by \autoref{Bij_SR}, 
	the map $\psi_\Lambda\colon \mathrm{Lat}_\mathrm{d}(R)\to\mathrm{Lat}(\Lambda(R))$ given by $\psi_\Lambda(I)=\Lambda_R(I)$ is a lattice isomorphism. Therefore, if we define $\phi\colon \mathrm{Lat}(\SQ(R))\to\mathrm{Lat}_\mathrm{d}(R)$, by $\phi(J,J\cap\SR(R))=\psi_\Lambda^{-1}(J)$, we see that $\phi$ is the inverse of $\psi$ and thus $\psi$ is a lattice isomorphism.
	
	Now, let $I$ be a quasipure ideal. By definition, $I$ is in particular decomposable. Let $[(x_{n})]\in \Lambda_R (I)$. Using \autoref{QP_Property}, we find for each $n\geq 1$ a sequence $(y_{n,m})_m\in\preS (I)$ such that $x_{n}\precsim_{1} y_{n,1}$, and thus $x_n\precsim_1y_{1,1}\oplus\cdots\oplus y_{n,1}$. This implies that $[(x_{n})]\leq \sup_{n} ([(y_{1,m})_m]+\ldots +[(y_{n,m})_m])$. Since $\SR (I)$ is a submonoid of $\Lambda_R(I)$ closed under suprema of (weakly) increasing sequences and $[(y_{n,m})_m]\in \SR(I)$ for each $n$, the above supremum is in $\SR (I)$. Thus, $\SR (I)$ is cofinal in $\Lambda_R (I)$.
	
	Conversely, assume that $I$ is decomposable and that $\SR (I)$ is cofinal in $\Lambda_R (I)$. Take any $x\in M_\infty(I)$, and by applying decomposability of $I$ choose a sequence $(x_n)\in M_\infty(I)$ such that $x=x_1$ and $x_n\precsim_1 x_{n+1}$ for all $n$. Since $[(x_n)]\in\Lambda_R(I)$ and $S(I)$ is cofinal,
	there is $[(y_n)]\in \SR(I)$ such that $[(x_1,x_2,\ldots )]\leq [(y_n)]$. Therefore there is $n$ with $x\precsim_{1}y_n$ and since $y_n\in \SR(I)$ there is $z_{n+1}\in M_\infty(R)$ such that $y_n=z_{n+1}y_{n+1}y_n$. Thus, with $y:=y_n\in M_\infty(I)$ and $s:=z_{n+1}y_{n+1}\in M_\infty(I)$, we have that $x\precsim_{1} y$ and $sy=y$. This proves that $I$ is quasipure.
	
	The last part of the statement follows using that any ideal in a weakly $s$-unital ring is decomposable.
\end{proof}
If $I$ is a decomposable ideal of a ring $R$, we set
$$\SQ _R(I) := (\Lambda _R(I), \SR (I)),$$
which is an ideal of $\SQ (R)$ by \autoref{thm:LatIsoIdeals}. 

\begin{remark}\label{rmk:IdCalg}
	The results in this and the previous sections apply, in particular, to C*-algebras. The structure of the non-closed ideals of those has been studied throughout the years (see, for example, \cite{Eff63, PedPet70}), but many fundamental questions still remain open. For example, it is not known whether every maximal ideal is closed. This is true in the unital case, but it remains an open problem in general. In the same vein, one can ask what are the trace ideals of projective modules over a C*-algebra $A$. In view of \autoref{Bij_WR}, this amounts to asking what are the ideals of $\SR(A)$. In this direction, it is shown in \cite[Theorem A]{gardella2023semiprimeidealscalgebras} that an ideal in a C*-algebra is idempotent if and only if it is semiprime. Since all trace ideals are idempotent, this implies that all trace ideals over a C*-algebra are semiprime. It is not hard to show that all Pedersen ideals of closed ideals of $A$ are quasipure and thus they are trace ideals of some projective module. Moreover, the converse holds for commutative $C^*$-algebras. 
\end{remark}


We now explore quotients and exactness in the categories $\SQ$ and $\SCu$.
\begin{lemma}\label{prp:CofinalQuot}
	Let $(S,W)\in \SQ$ and let $(I,I\cap W)$ be an ideal of $(S,W)$. If $I\cap W\subseteq I$ is cofinal, then $(S/I,W/I\cap W)$ is an object in $\SQ$.
\end{lemma}
\begin{proof}
	Let $x,z\in W$ be such that $x\leq_{I} z$ in $S$. Since $W\cap I$ is cofinal in $I$, it follows that $x\leq z+y$ for some $y\in W\cap I$. Therefore, $x\leq_{W\cap I} z$. This implies that $W/I\cap W$ order-embeds into $S/I$.
	
	To see that $W/I\cap W$ is closed under weakly increasing sequences, let $([x_{n}])_n$ be a weakly increasing sequence in $S/I$ with $x_{n}\in W$. By definition, 
	there exists an increasing sequence $([z_{k}])_k$ in $S/I$ satisfying: 
	\begin{enumerate}[(i)]
		\item For every $k$ there exists $n_{k}$ such that $[z_{k}]\leq [x_{n}]$ for every $n\geq n_{k}$.
		\item $[x_{n}]\leq \sup_{k} [z_{k}]$ for every $n$.
	\end{enumerate}
	
	Without loss of generality, we may assume that $(z_{k})_k$ is increasing in $S$. Let $z$ be its supremum. Since $W\cap I$ is cofinal in $I$, one gets an increasing sequence $(n_k)$ of positive integers such that:
	\begin{enumerate}[(i)']
		\item For every $k$ there exists $y_{k}\in I\cap W$ such that $z_{k}\leq x_{n_k}+y_{k}$.
		\item For every $n$, there exists $\tilde{y}_{n}\in I\cap W$ such that $x_{n}\leq z+\tilde{y}_{n}$.
	\end{enumerate}
	
	Consider the following elements in $W\cap I$:
	\[
	S_{k}:=\sum_{i=1}^{k}y_{i},\quad\text{and}\quad \tilde{S}_{k}:=\sum_{i=1}^{k}\tilde{y}_{i}.
	\]
	
	Note that
	\[
	z_{k}+S_{k-1}+\tilde{S}_{k-1}\leq x_{n_k}+S_{k}+\tilde{S}_{k-1}\leq z+S_{k}+\tilde{S}_{n_k}
	\]
	for each $k\ge 1$. 
	
	Denote by $S_\infty$ and $\tilde{S}_{\infty}$ the suprema of $S_k$ and $\tilde{S}_k$ respectively. The sequence $(x_{n_k}+S_{k}+\tilde{S}_{k-1})_k$ satisfies (i) and (ii) in \autoref{par:wincr} with respect to the increasing sequence $(z_{k}+S_{k-1}+\tilde{S}_{k-1})_k$, and thus it is weakly increasing in $S$. Indeed, for $l\ge k$ we have
	$$z_k+S_{k-1}+\tilde{S}_{k-1} \le z_l +S_{k-1} +\tilde{S}_{k-1} \le x_{n_l} + S_l +\tilde{S}_{l-1},$$
	and for all $k\ge 1$ we have
	$$x_{n_k}+S_k+\tilde{S}_{k-1} \le z+S_{k}+\tilde{S}_{n_k} \le z+S_{\infty}+ \tilde{S}_{\infty} = \sup (z_k+S_{k-1}+\tilde{S}_{k-1}).$$
	
	Moreover, since the elements of the sequence are in $W$ and $W\subseteq S$ is closed under suprema of weakly increasing sequences, we have 
	\[
	z+S_{\infty}+\tilde{S}_{\infty} = \sup_n (x_n+S_{\infty}+\tilde{S}_{\infty})\in W.
	\]
	
	This implies that $[z]\in W/W\cap I$ as required.
\end{proof}
\begin{parag}[Exact sequences in the categories $\SQ$ and $\SCu$]
	Let $\varphi\colon M\rightarrow N$ be a morphism of positively ordered monoids. As in 
	\cite{Ciuperca2010}, we define
	\[
	\begin{split}
		{\rm Im}(\varphi) &=\{ (h_{1},h_{2})\in N\times N\mid h_{1}\leq \varphi (s)+h_{2}\text{ for some }s\in M \},\\
		{\rm ker}(\varphi) &=\{ (s_{1},s_{2})\in M\times M\mid \varphi (s_{1})\leq \varphi (s_{2}) \},
	\end{split}
	\]
	and we say that a sequence 
	\[
	0\rightarrow 
	(I,J) \rightarrow
	(S,W) \rightarrow
	(Z,T) \rightarrow
	0
	\]
	in $\SQ$ (respectively, in $\SCu$) is \emph{exact} if
	\[
	0\rightarrow 
	I \rightarrow
	S \rightarrow
	Z \rightarrow
	0,
	\quad\text{and}\quad 
	0\rightarrow 
	J \rightarrow
	W \rightarrow
	T \rightarrow
	0
	\]
	are exact in the standard sense using the above definitions of image and kernel.
\end{parag}
\begin{proposition}\label{prp:ExactSeq}
	Let $(S,W)\in \SQ$  and let $(I,I\cap W)$ be an ideal of $(S,W)$. Assume that $I\cap W\subseteq I$ is cofinal. Then, the sequence
	\[
	0\rightarrow (I,I\cap W)\rightarrow (S,W)\rightarrow 
	(S/I,W/I\cap W)
	\rightarrow 0
	\]
	is exact.
\end{proposition}
\begin{proof}
	The first component is exact by the same argument as in \cite[Theorem 4.1]{Ciuperca2010}. 
	
	To see that $0\to I\cap W\to W\to W/I\cap W\to 0$ is exact, note that 
	\[
	{\rm ker}( W\to W/I\cap W ) = 
	{\rm ker}( S\to S/I )\cap (W\times W)
	\]
	and 
	\[
	{\rm Im}(I\cap W\to W)\subseteq {\rm Im}(I\to S).
	\]
	
	Thus, using that ${\rm ker}( S\to S/I ) = {\rm Im}(I\to S)$, one gets 
	\[
	{\rm Im}(I\cap W\to W) \subseteq 
	{\rm ker}( W\to W/I\cap W ).
	\]
	
	To see the converse inclusion, take $(r_1,r_2)\in W\times W$ such that $[r_1]\leq [r_2]$ in $W/I\cap W$. Then, there exists $s\in I\cap W$ such that $r_1\leq r_2+s$. This shows that $(r_1,r_2)\in {\rm Im}(I\cap W\to W)$, as desired.
\end{proof}
\begin{theorem}\label{prp:ExactSCu}
	Let $I$ be a two-sided, quasipure ideal of a ring $R$. Then, the sequence 
	\[
	0\rightarrow \SQ_R (I)\rightarrow \SQ (R)\rightarrow \SQ (R/I)\rightarrow 0
	\]
	is exact. 
\end{theorem}
\begin{proof}
	This follows immediately from \autoref{thm:quo} and \autoref{prp:ExactSeq}.
\end{proof}
\begin{remark}
	The results in this section apply almost verbatim to the category $\SCu$ except for this last result, for which we will need to restrict to a particular class of rings; see \autoref{rmk:exactscu}.  
\end{remark}
\section{Dense and left normal rings}
\label{sec:leftdense}
In this section we introduce the notions of dense and left normal rings. We show that they constitute a large class for which the functors $\Lambda(\text{--})$ and $\SR(\text{--})$ are well-behaved; see \autoref{thm:dense}. 

\begin{parag}[Dense rings]
	\label{def:densering}
	Recall that a relation $\prec$ on a set $X$ is termed \emph{dense} (or also \emph{idempotent}) if for any $x,y\in X$ with $x\prec y$, there exists $z\in X$ with $x\prec z\prec y$. 
	
	We say that a ring $R$ is \emph{dense} provided the relation $\precsim_1$ is dense on $M_\infty(R)$.	If $R$ is a weakly $s$-unital ring (in particular, if it is unital) then the relation $\precsim_1$ on $M_\infty(R)$ is dense, simply because it is reflexive. This also holds more generally, for example, when $R$ is idempotent. Indeed, if $x\precsim_1 y$ in $M_\infty(R)$, write $x=ayb$ and use that $R$ is idempotent to decompose $a=a_1'a_2$ and $b=b_2b_1'$ in $M_\infty(R)$. Let $z=a_2yb_2$, and then we have $x=a_1'(a_2yb_2)b_1'\precsim_1z\precsim_1 y$.
\end{parag}

\begin{parag}[Left normal rings]
	\label{def:normalring}
	Let $R$ be a ring. We say that $R$ is \emph{left normal} if, for every $a,b,c\in M_\infty (R)$ such that 
	\[
	a=ba,\quad\text{and}\quad b=cb,
	\]
	there exist $d,e\in M_\infty (R)$ such that
	\[
	a=da,\quad 
	d=ed,\quad\text{and}\quad
	e=ce.
	\]	
	We give below some examples of left normal rings.
\end{parag}

Recall that a unital ring $R$ is \emph{weakly semihereditary} if for any $R$-linear maps $f\colon A\to B$ and $g\colon B\to C$ between finitely generated projective modules such that $g\circ f=0$, there is a decomposition $B=B'\oplus B''$ such that $\mathrm{im}(f)\subseteq B'\subseteq \ker(g)$. This is a right-left symmetric notion introduced by G. M. Bergman and satisfied by every right hereditary ring; see for example \cite[Part 1.11]{DickSch88}. Right semihereditary rings are also weakly semihereditary, and thus this class contains in particular all von Neumann regular rings as well as the path $K$-algebra of a quiver (where $K$ is a field).
\begin{lemma}
	Any weakly semihereditary ring is left normal.
\end{lemma}
\begin{proof}
	Let $R$ be as in the statement, and let $a,b,c\in M_\infty(R)$ be such that $a=ba$ and $b=cb$. By passing to matrices over $R$, we may assume that the elements are in $R$.
	
	Consider the $R$-linear maps $f\colon R\to R$ and $g\colon R\to R$ given by $f(x)=ax$ and $g(y)=(1-b)y$. Clearly $g\circ f=0$, hence by assumption there is an idempotent $e\in R$ such that $R=eR\oplus (1-e)R$ and $a=ea$, while $(1-b)e=0$. Let $d=e$, and note that $e=be=cbe=ce$. Thus $R$ is left normal.
\end{proof}	

The definition and terminology of left normality is motivated from topology.
\begin{lemma}
	\label{lma:normal}
	Let $R=C(X, \mathbb{K})$ with $X$ a normal space and $\mathbb{K}=\R$ or $\mathbb{C}$. Let $f,g,h\in  R$ be such that $f=gf$ and $g=hg$. Then, there exist $f',g'\in R$ such that $f=f'f$, $f'=g'f'$ and $g'=hg'$.
\end{lemma}
\begin{proof}
	Let $f,g,h\in C(X, \mathbb{K})$ be as in the statement.
	Consider the closed set $C=\mathrm{supp}(f)$ and the open set $V=\mathrm{coz}(g)$, where $\mathrm{coz}(g)$ denotes the cozero set of $g$ and $\mathrm{supp}(f)$ denotes the closure of $\mathrm{coz}(f)$.
	
	Put $D= \{ t\in X : g(t) > 1/2\}$, which is an open subset of $X$. We have $C\subseteq D$ and $\ol{D}\subseteq V$.	
	
	By Urysohn's Lemma, there are $f',g'\in C(X,\mathbb{K})$ such that $0\leq f',g'\leq 1$, $f'$ is $1$ on $C$ and $0$ out of $D$, and $g'$ is $1$ on $\ol{D}$ and $0$ out of $V$. 
	Then $f',g'$ satisfy the desired conditions.
\end{proof}
\begin{remark}
	It is plausible that \autoref{lma:normal} can be improved to show that $C(X)$ is, at least in the complex case, a left normal ring. However a proof has not come around yet.
\end{remark}	
Another instance of left normality comes from C*-algebra theory. Recall that a C*-algebra $A$ is a complex Banach algebra with involution  such that $\Vert aa^*\Vert=\Vert a\Vert^2$ for any $a\in A$. Important elements in C*-algebras are the so-called \emph{positive} elements, that is, those of the form $x^*x$.

The subclass of \emph{SAW*-algebras}, introduced by Pedersen in \cite{Ped86}, plays an important role in the study of multiplier and corona algebras. We recall the definition here for convenience. A C*-algebra $A$ is an SAW*-algebra if for any given positive elements $x,y\in A$ such that $xy=0$, there is a positive element $e\in A$ such that $ex=x$ and $ey=0$. (Because of the involution, note that this also implies $xe=x$ and $ye=0$.) It was proved in \cite[Theorem 13]{Ped86} that the corona algebra of any $\sigma$-unital (in particular, of any separable) C*-algebra is an SAW*-algebra. It is an open problem to decide whether SAW*-algebras are closed under the passage to matrices, alghough this is known in some significant cases. For example, if $A$ is a $\sigma$-unital C*-algebra, $\mathcal{M}(A)$ is its multiplier algebra, and $\mathcal{M}(A)/A$ is the corona algebra, then $M_n(\mathcal{M}(A))\cong \mathcal{M} (M_n(A))$ for all $n\geq 1$, and therefore $M_n(\mathcal{M}(A)/A)\cong \mathcal{M}(M_n(A))/M_n(A)$. Another example is constituted by the class of Rickart C*-algebras, which are also SAW*-algebras and that were shown to be matrix stable in \cite[Theorem 3.4]{AraGol93}.


\begin{proposition}
	\label{prp:saw}
	Let $A$ be a $C^*$-algebra such that $M_n(A)$ is an SAW*-algebra for all $n\ge 1$. Then $A$ is a left normal ring.
\end{proposition}	
\begin{proof}
	Let $a,b,c\in M_{\infty}(A)$ such that $a=ba$ and $b=cb$. We may assume that $a,b,c\in M_{n}(A)$ for some $n$, which by assumption is an SAW*-algebra, hence without loss of generality $a,b,c\in A$.
	
	Then $aa^*=baa^*$ and $bb^*=cbb^*$. From the first equality we get $(1-b)aa^*=0$, whence also $(1-b)^*(1-b)aa^*=0$. Thus there exists a positive element $d\in A$ such that $daa^*=aa^*$ and $(1-b)^*(1-b)d=0$.
	
	Using the first equality we get
	\[
	0=\Vert (d-1)aa^*(d-1)^*\Vert=\Vert (d-1)a\Vert^2,
	\]
	whence $a=da$. Similarly, the second equality yields $bd=d$.
	
	Now, since $(1-c)bb^*=0$, we have $(1-c)^*(1-c)bb^*=0$, hence there exists a positive element $e\in A$ such that $ebb^*=bb^*$ and $(1-c)^*(1-c)e=0$. Thus, arguing as before, we get $eb=b$ and $(1-c)e=0$, that is, $e=ce$.
	
	Finally, $ed=ebd=bd=d$, as required.
\end{proof}	

The relevance of considering dense and left normal rings is reflected in Theorem~\ref{thm:dense} below. First we need the following lemma.
\begin{lemma}
	\label{lem:easylemma}
	Let $R$ be a left normal ring, and let $a,b,c\in M_\infty (R)$ be such that
	\[
	a=ba, \quad \text{and}\quad b=cb.
	\]
	
	Then, there exists a sequence $(d_n)$ in $M_\infty (R)$ such that 
	\[
	a=d_1a,\quad  
	d_n= d_{n+1}d_n,\quad\text{and}\quad
	d_n=cd_n
	\]
	for all $n$.
\end{lemma}
\begin{proof} 
	It follows from the definition of left normal ring that there are $d_1,e_1\in M_\infty (R)$ such that 
	\[
	a=d_1a,\quad 
	d_1= e_1 d_{1},\quad \text{and}\quad 
	e_1=ce_1.
	\]
	
	Hence $d_1= e_1d_1 = ce_1d_1= cd_1$. 
	
	Now, proceeding by induction, let $n\geq 1$ and assume that we have found  elements  $d_1,\ldots, d_n, e_n$ in $M_\infty (R)$ such that
	\[
	a=d_1a,\quad 
	d_n=e_n d_n,\quad 
	e_n= c e_n,\quad \text{and}\quad
	d_i= d_{i+1}d_i
	\] 
	for $i=1,\ldots, n-1$.
	
	Note that these conditions imply that $cd_i=d_i$ for all $i=1,\dots ,n$. Indeed, we have
	\[
	cd_{n}= ce_{n}d_{n} = e_{n}d_{n}= d_{n}
	\]
	and, using that $d_i=d_{i+1}d_i$, one also gets $cd_i=d_i$ for each $i$.
	
	Using that $R$ is left normal once again with $d_n,e_n,c$ we get $d_{n+1},e_{n+1}\in M_\infty (R)$ such that
	\[
	d_n= d_{n+1}d_n,\quad 
	d_{n+1}= e_{n+1}d_{n+1},\quad \text{and}\quad 
	e_{n+1}=ce_{n+1},
	\]
	thus completing the inductive argument, and the proof.
\end{proof}
Recall that the semigroups $\Lambda(R)$ and $\SR(R)$ are equipped with an auxiliary relation $\prec$ defined as $[(x_n)]\prec [(y_n)]$ if there is $m$ such that $x_n\precsim_1 y_m$ for all $m$; see \autoref{rmk:prec}. As we show below, this is identified with the way-below relation in relevant cases.
\begin{theorem}
	\label{thm:dense}
	Let $R$ be a ring.
	\begin{enumerate}[{\rm (i)}]
		\item If $R$ is  dense, then $\Lambda(R)$ is a $\Cu$-semigroup, and $\prec = \ll$ on $\Lambda (R)$.
		\item If $R$ is left normal, then $\SR(R)$ is a $\Cu$-semigroup, and $\prec = \ll$ on $\SR (R)$.
	\end{enumerate}
\end{theorem}
\begin{proof}
	In both (i) and (ii) we already know, irrespective of other assumptions, that $\Lambda(R)$ and $\SR(R) $ are $\Q$-semigroups, and thus satisfy axioms (O1) and (O4). 
	
	(i): Let $[(x_n)]\in \Lambda(R)$, and let $z_1^{(n)}=x_n$. Since $\precsim_1$ is dense, there is a sequence $z_k^{(n)}$ such that $x_n\precsim_1 z_k^{(n)}\precsim_1z_{k+1}^{(n)}\precsim_1 x_{n+1}$ for all $n$ and $k\geq 2$. Define $z_n=[(z_k^{(n)})_k]$ and note that by construction $z_m \prec z_{m+1}\prec [(x_n)]$ for each $m$.
	
	We now claim that $[(x_n)]=\sup_m z_m$. To verify this, let us briefly recall the details on how suprema are constructed in $\Lambda(R)$ (see \cite[Proposition 2.13]{AntAraBosPerVil23:CuRing} and also \cite[Proposition 3.1.6]{APT-Memoirs2018}).  By an inductive process, one may choose an increasing sequence $(m_k)$ such that $z_{m_i+j}^{(i)}\precsim_1 z_{m_k}^{(k)}$ whenever $i+j\leq k$. Then the sequence $(z_{m_k}^{(k)})_k$ defines an element in $\Lambda(R)$ and $\sup z_m=[(z_{m_k}^{(k)})]$. Now, for each $n$, we have that $x_n=z_1^{(n)}\precsim_1 z_{m_n}^{(n)}$, and thus $[(x_n)]\leq [(z_{m_k}^{(k)})]$. Conversely, for each $k$, we have that $z_{m_k}^{(k)}\precsim_1 x_{k+1}$, thus establishing the claim.
	
	Since $\prec$ is compatible with addition, it only remains to show that $\prec$ agrees with the compact containment relation. As observed in \autoref{par:QCu}, $\prec$ is always stronger than $\ll$. Hence, assume that $[(x_n)]\ll [(y_n)]$. By the first part of the proof, $[(y_n)]=\sup w_m$, where $w_m=[(w_k^{(m)})_k]$ satisfying $w_1^{(m)}=y_m$ and $w_k^{(m)}\precsim_1 y_{m+1}$ for all $k$. Then, there is $m$ such that $[(x_n)]\leq w_m$ and therefore, for each $n$, there is $k$ with $x_n\precsim_1 w_k^{(m)}\precsim_1 y_{m+1}$. This implies that $[(x_n)]\prec [(y_m)]$.	
	
	(ii): Given $[(x_n)]\in \SR(R)$, we first use \cite[Corollary 4.11]{AntAraBosPerVil23:CuRing} to assume without loss of generality that $x_{n+1}x_n=x_n$ for all $n$. (Note that we still have $x_n\precsim_1 x_{n+1}$, since $x_n=x_{n+1}x_n=x_{n+2}x_{n+1}x_n$.)
	
	Since for each $n$ we have $x_n=x_{n+1}x_n$ and $x_{n+1}=x_{n+2}x_{n+1}$ we may apply \autoref{lem:easylemma} to find a sequence $(z_k^{(n)})_k$ such that $x_n=z_1^{(n)}x_n$, $z_k^{(n)}=z_{k+1}^{(n)}z_k^{(n)}$, and $z_k^{(n)}=x_{n+2}z_k^{(n)}$ for each $k$. Using that $x_{n+2}=x_{n+3}x_{n+2}$ we have
	\[
	x_n\precsim_1 z_k^{(n)}\precsim_1 z_{k+1}^{(n)}\precsim_1 x_{n+2}
	\]
	for each $k,n$. We may assume, after reindexing, that $x_n\precsim_1 z_k^{(n)}\precsim_1 z_{k+1}^{(n)}\precsim_1 x_{n+1}$.
	
	Therefore $z_n:=[(z_k^{(n)})_k]\in\SR(R)$ and $z_m\prec z_{m+1}\prec [(x_n)]$ for each $m$. By (the proof of) \cite[Lemma 4.3]{AntAraBosPerVil23:CuRing}, the supremum of an increasing sequence in $\SR(R)$ agrees with the supremum of the same sequence computed in $\Lambda(R)$. Therefore, we may proceed as in the proof of (i) to conclude that $[(x_n)]=\sup_m z_m$.
	
	Now the same argument used in (i) shows that $\prec$ agrees with the relation $\ll$ on $\SR (R)$. This finishes the proof.
\end{proof}
\begin{remark}
	Observe that the notion of left normality provides the appropriate density condition for the semigroup $\SR(R)$ to be in $\Cu$, but that there is in general no apparent connection between density and left normality.	
\end{remark}	
Let us denote by $\mathrm{Rings}^{dense}$ the full subcategory of $\mathrm{Rings}$ whose objects are dense rings. 
\begin{corollary}
	Let $R$ be a dense ring, then $\SCu (R)=(\Lambda (R), \SR (R))$ is an object in $\SCu$.~Further, the assignment $R\mapsto \SCu(R)$ defines a functor $\SCu\colon \mathrm{Rings}^{dense}\to \SCu$, and we have a commutative diagram
	\[
	\xymatrix{
		\mathrm{Rings} \ar@{->}[r]^{\SQ} & \SQ       \\
		\mathrm{Rings}^{dense}  \ar@{^{(}->}[u]^\iota \ar@{->}[r]^{\quad\!\!\!\SCu}	&    \SCu\ar@{^{(}->}[u]_\iota }  
	\]
	where $\iota$ stands for the respective inclusion functors.
\end{corollary}
\begin{proof}
	The first part of the statement follows from \autoref{thm:dense}, while the second part is clear from \autoref{thm:FunctorSQ}, and by construction.
\end{proof}
\begin{remark}
	\label{rmk:exactscu}
	Note that, if $I$ is a decomposable ideal of a dense ring $R$, then $\SCu_R (I) = (\Lambda_R(I), \SR (I))$ is an ideal of $\SCu (R)$; see \autoref{IdealsSQ}. Further, restricting to the subcategory of dense rings, we have have by \autoref{prp:ExactSCu} that the sequence 
	\[
	0\rightarrow \SCu_R (I)\rightarrow \SCu (R)\rightarrow \SCu (R/I)\rightarrow 0
	\]
	is exact.
\end{remark}
\section{Inductive limits and continuity}
\label{sec:IndLim}
In this section we show that, for the class of dense rings, the assignment $R\mapsto\Lambda(R)$ defines a continous functor. Similarly, for the class of left normal rings, the assignment $R\mapsto \SR (R)$ defines a continuous functor. We start by recalling how limits in the category $\Cu$ are constructed.

\begin{parag}[Limits in $\Cu$]\label{pgr:LimCu}
	Let $((S_{\lambda})_{\lambda\in\Omega},(f_{\mu,\lambda})_{\mu\geq\lambda})$ be a direct system of $\Cu$-semigroups, that is, each $S_\lambda$ is a $\Cu$-semigroup for each $\lambda$ and, for every pair $\lambda,\mu$ with $\mu\geq \lambda$, there exists a $\Cu$-morphism $f_{\mu,\lambda}\colon S_{\lambda}\to S_{\mu}$ such that
	$f_{\lambda ,\lambda }=\id$ and $f_{\mu,\lambda}\circ f_{\lambda,\nu}=f_{\mu,\nu}$ whenever $\nu\leq\lambda\leq \mu$ in $\Omega$.
	
	By \cite[Corollary~3.1.11]{APT-Memoirs2018} (see also \cite[Theorem~2]{Coward2008}), the system has a direct limit $\lim_\lambda S_\lambda$ in the category $\Cu$. We will denote by $f_\lambda$ the canonical maps $f_\lambda\colon S_\lambda\to\lim_\lambda S_\lambda$ given by the induced limit in $\Cu$.
	
	As shown in \cite[Lemma~3.8]{ThiVil22DimCu} (see also \cite{Coward2008}), a $\Cu$-semigroup $S$ together with maps $f_\lambda\colon S_\lambda\to S$ is the limit of the system above if and only if the following conditions hold:
	\begin{enumerate}[(a)]
		\item $f_\mu\circ f_{\mu,\lambda}=f_\lambda$ whenever $\mu \geq \lambda$;
		\item for any pair $x'\ll x$ in $S_\lambda$ and an element $y\in S_\mu$ such that $f_\lambda (x)\leq f_\mu (y)$, there exists $\nu\geq \mu,\lambda$ such that $f_{\nu, \lambda } (x')\ll f_{\nu, \mu} (y)$;
		\item for every pair $x'\ll x$ in $S$ there exists $y\in S_\lambda$ such that $x'\leq f_\lambda (y)\leq x$.
	\end{enumerate}
\end{parag}
\begin{theorem}
	\label{thm:limitlambda}
	Let $((R_{\lambda})_{\lambda\in\Omega},(\phi_{\mu,\lambda})_{\mu\geq\lambda})$ be a direct system in $\mathrm{Rings}$.
	\begin{enumerate}[{\rm (i)}]
		\item If all $R_\lambda$ are dense (this is the case, for example, if $R_\lambda^2=R_\lambda$ for all $\lambda$), then  $\varinjlim R_\lambda$ is also dense and
		\[
		\Lambda(\varinjlim R_\lambda)=\varinjlim\limits_{\Cu}\Lambda(R_\lambda).
		\]
		\item If all $R_\lambda$ are left normal, then $\varinjlim R_\lambda$ is also left normal and
		\[
		\SR(\varinjlim R_\lambda)=\varinjlim\limits_{\Cu}\SR(R_\lambda).
		\]
	\end{enumerate}
\end{theorem}
\begin{proof}
	Write $R=\varinjlim R_\lambda$ and denote by $\phi_\lambda\colon R_\lambda\to R$ the limit maps. Suppose that $x\precsim_1 y$ in $M_\infty(R)$. Then there is $\lambda\in \Omega$ such that $x,y\in M _\infty(R_\lambda)$ and $x\precsim_1 y$ in $M _\infty(R_\lambda)$. This shows that $\precsim_1$ is dense in $M_\infty(R)$. It is also easily checked that, if each $R_\lambda$ is left normal, so is $R$.
	
	(i): For each $\mu\geq\lambda$ in $\Omega$, denote $f_{\mu,\lambda}=\Lambda(\phi_{\mu,\lambda})$ and $f_\lambda=\Lambda(\phi_\lambda)$.
	To check that $\Lambda(R)$ is the limit of the system $(\Lambda(R_\lambda)_\lambda,(f_{\mu,\lambda})_{\mu\geq\lambda})$, we note that $\Lambda(R_\lambda)$ and $\Lambda(R)$ are all $\Cu$-semigroups by (i) in \autoref{thm:dense} and thus we may use the characterization given in \autoref{pgr:LimCu}. Note also that, by the proof of said theorem, the relations $\prec$ and $\ll$ in $\Lambda(R)$ and in any of the $\Lambda(R_\lambda)$ agree. Condition (a) is already satisfied by definition of the maps $f_\lambda$ and $f_{\mu,\lambda}$.
	
	To check condition (b), let $[(x_n)], [(z_n)]\in\Lambda(R_\lambda)$ with $[(z_n)]\prec [(x_n)]$, and let $[(y_n)]\in\Lambda(R_\mu)$. Assume that $f_\lambda([(x_n)])=[(\phi_\lambda(x_n))]\leq[(\phi_\mu(y_n))]= f_\mu([(y_n)])$. We know from the first assumption that there is  $m$ such that $z_n\precsim_1 x_m$ for all $n$. Thus, for $m$ as above there is $l$ such that $\phi_\lambda(x_m)\precsim_1 \phi_\mu(y_l)$ in $M_\infty(R)$. This means that there is $\nu\geq\lambda,\mu$ for which $\phi_{\nu,\lambda}(x_m)\precsim_1\phi_{\nu,\mu}(y_l)$ in $M_\infty(R_{\nu})$. Therefore, $\phi_{\nu,\lambda}(z_n)\precsim_1 \phi_{\nu,\lambda}(x_m)\precsim_1\phi_{\nu,\mu}(y_l)$ for all $n$, and this implies that $f_{\nu,\lambda}([(z_n)])\prec f_{\nu,\mu}([(y_n)])$.
	
	Finally, let us check condition (c). Take $[(x_n)]\prec[(y_n)]$ in $\Lambda(R)$. This implies that there is $m$ such that $x_n\precsim_1 y_m\precsim_1 y_{m+1}$. We may assume that there is $\lambda$ such that  $y_m=\phi_\lambda(y_m'),y_{m+1}=\phi_\lambda(y_{m+1}')$ with $y_m'\precsim_1y_{m+1}'$ in $M_\infty(R_\lambda)$. Since the relation $\precsim_1$ in $M_\infty(R_\lambda)$ is dense by assumption, there is a sequence $(z_k)$ in $M_\infty(R_\lambda)$ such that $y_m'\precsim_1z_k\precsim_1 z_{k+1}\precsim_1 y_{m+1}'$ for all $k$. Now $[(x_n)]\leq f_\lambda[(z_n)]\leq [(y_n)]$, as desired.
	
	(ii): In analogy with (i), for each $\mu\geq\lambda$ in $\Omega$, denote $g_{\mu,\lambda}=\SR(\phi_{\mu,\lambda})$ and $g_\lambda=\SR(\phi_\lambda)$.
	To check that $\SR(R)$ is the limit of the system $(\SR(R_\lambda)_\lambda,(g_{\mu,\lambda})_{\mu\geq\lambda})$, again we use the characterization given in \autoref{pgr:LimCu}:
	
	First note that $\SR(R_\lambda)$ and $\SR(R)$ are all $\Cu$-semigroups by (ii) in \autoref{thm:dense}. It is automatic that the maps $g_{\mu,\lambda},g_\lambda$ satisfy condition (a).
	
	Recall from the proof of \autoref{thm:dense} that the relations $\prec$ and $\ll$ coalesce in $\SR(R_\lambda)$ and $\SR(R)$ since by assumption they are left normal rings. Then, the proof for (b) follows verbatim as in (i) above.
	
	Finally, for (c), assume that $[(x_n)]\prec[(y_n)]$ in $\SR(R)$, where we may assume by \cite[Corollary 4.11]{AntAraBosPerVil23:CuRing} that $y_{n+1}y_n=y_n$ for all $n$. This implies that there is $m$ such that $x_n\precsim_1 y_m$ for all $n$. We may also assume that there is $\lambda$ such that  $y_m=\phi_\lambda(y_m'),y_{m+1}=\phi_\lambda(y_{m+1}'), y_{m+2}=\phi_\lambda(y_{m+2}')$, and $y_{m+3}=\phi_\lambda(y_{m+3}')$ with $y_m'=y_{m+1}'y_m',y_{m+1}'=y_{m+2}'y_{m+1}'$, and $y_{m+2}'=y_{m+3}'y_{m+2}'$  in $M_\infty(R_\lambda)$.
	
	Since $R_\lambda$ is a left normal ring by assumption, we may apply \autoref{lem:easylemma} and find a sequence $(z_k^{(m)})$ such that $y_m'=z_1^{(m)}y_m'$, $z_k^{(m)}=z_{k+1}^{(m)}z_k^{(m)}$, and $z_k^{(m)}=y_{m+2}z_k^{(m)}$ for each $k$. Hence $[(z_k^{(m}))]\in \SR(R_\lambda)$ and we have that $[(x_n)]\leq g_\lambda([z_k^{(m)}])\leq [(y_m)]$.		
\end{proof}
When considering weakly $s$-unital rings, the result in \autoref{thm:limitlambda} (i) may be expressed in terms of the semigroup $\W(R)$; see the discussion in \autoref{par:WRetal}. 
\begin{parag}[Intervals and algebraic $\Cu$-semigroups]
	\label{par:algebraic}
	Recall that a countably generated interval $I$ in a positively ordered semigroup $M$ is an upward directed, order-hereditary subset $I$ of $M$ that has a countable cofinal subset. We denote by $\Lambda_\sigma(M)$ the collection of countably generated intervals. This is also a positively ordered semigroup with order induced by set inclusion and addition defined as $I+J=\{x\in M\colon x\leq y+z\text{ where }y\in I,z\in J\}$; see e.g. \cite[Section 5.5]{APT-Memoirs2018}. As already mentioned in \autoref{par:WRetal}, if $R$ is a weakly $s$-unital ring, then 
	$\Lambda_{\sigma}(\W(R))$ may be identified with $\Lambda(R)$.
	
	Recall that an element $x$ in an  ordered semigroup $S$ satisfying (O1) is termed \emph{compact} provided that $x\ll x$. The submonoid of compact elements of $S$ is denoted by $S_c$. For a $\Cu$-semigroup of the form $\Cu(A)$ of a C*-algebra $A$, the natural compact elements (and very often the only ones) have the form $x=[p]$ where $p$ is a self-adjoint idempotent (a projection). We say that a $\Cu$-semigroup $S$ is \emph{algebraic} provided every element in $S$ is the supremum of an increasing sequence of compact elements. Examples, coming from C*-theory, of algebraic $\Cu$-semigroups include, for example, the Cuntz semigroup of any C*-algebra which, as a ring, is an exchange ring; see \cite[Remark 5.5.2(2)]{APT-Memoirs2018} and \cite[Theorem 7.2]{AraGooOmPar98}. In connection with the discussion above, if $M$ is any positively ordered monoid, then the monoid $\Lambda_\sigma(M)$ is an algebraic $\Cu$-semigroup, with $\Lambda_\sigma(M)_c\cong M$ (see \cite[Lemma 2.15]{AntAraBosPerVil23:CuRing}).~In fact, each interval $I$ generated by a countable cofinal increasing sequence $(x_n)$ may be written as $I=\sup [0,x_n]$, where clearly $[0,x_n]\ll [0,x_n]$ for each $n$. 
\end{parag}	
\begin{parag}[Limits in the category $\POM$]
	Given a direct system $((M_\lambda)_{\lambda\in\Omega}, (f_{\mu ,\lambda})_{\mu\geq\lambda})$ in $\POM$ over a directed set $\Omega$, recall that its direct limit in $\POM$ may be constructed as the algebraic limit $(M,(f_\lambda)_{\lambda\in\Omega})$ of the system, where $f_\lambda\colon M_\lambda\to M$, equipped with the usual addition and `asymptotic' order, that is, $f_\lambda (x )\leq f_\mu (y)$ in $M$ for $x\in M_\lambda$ and $y\in M_\mu$ if there exists $\delta\geq \lambda ,\mu$ such that $f_{\delta ,\lambda }(x)\leq f_{\delta ,\mu }(y)$ in $M_\delta$. We write $ \lim\limits_{\POM} (M_\lambda, f_{\mu,\lambda})$, or just $\lim\limits_{\POM} M_\lambda$. In the following, we denote by Rings$^{ws}$ the category of weakly s-unital rings and ring homomorphisms.
\end{parag}	
\begin{proposition}\label{prp:WSeqCont}
	Let $((M_{\lambda})_{\lambda\in\Omega},(f_{\mu,\lambda})_{\mu\geq\lambda})$ be a direct system in $\POM$ and let $((R_{\lambda})_{\lambda\in\Omega},(\phi_{\mu,\lambda})_{\mu\geq\lambda})$ be a direct system in $\mathrm{Rings}^{ws}$. Then, 
	\[
	\Lambda_\sigma (\lim_{\POM} M_\lambda ) \cong \lim_{\Cu} \Lambda_\sigma (M_\lambda)
	\quad\text{ and }\quad 	\W (\lim R_\lambda ) \cong \lim_{\POM} \W (R_\lambda).
	\]
\end{proposition}
\begin{proof}
	It was proved in \cite[Proposition 5.5.5 and Remark 5.5.6]{APT-Memoirs2018} that the correspondence $M\mapsto \Lambda_\sigma (M)$ extends to a functor $\POM \to \Cu_\mathrm{alg}$ that yields an equivalence between these categories (via the functor $\Cu_{\mathrm{alg}}\to \POM$ given by $S\mapsto S_c$). One furthermore gets a bijection between the morphism sets  
	\[
	\Cu(\Lambda_\sigma(M),\Lambda_\sigma(N))\cong\POM(M,N).
	\]
	With this at hand, in combination with \cite[Proposition 3.1.6 and Theorem 3.1.8]{APT-Memoirs2018} (which provides us with a bijection $\Cu(\Lambda_\sigma(M),S)\cong \POM(M,S_c)$ for any $\Cu$-se\-mi\-group $S$), one gets $\Lambda_\sigma(\lim\limits_{\POM}M_\lambda)\cong\lim\limits_{\Cu}\Lambda_\sigma(M_\lambda)$, thus establishing the leftmost isomorphism.
	
	For the rightmost isomorphism, we first apply the functor $\W(\text{--})$ to the system $((R_{\lambda})_{\lambda\in\Omega},(\phi_{\mu,\lambda})_{\mu\geq\lambda})$ and its limit $(\lim R_\lambda, (\phi_\lambda)_{\lambda\in\Omega})$ to obtain a direct system in the category $\POM$ and  a $\POM$-morphism $\varphi\colon \lim\limits_{\mathrm{POM}}\W(R_\lambda)\to\W(\lim R_\lambda)$ such that the following diagram commutes:
	\[
	\xymatrix{
		\W (R_\lambda) \ar@{->}[r]   \ar@{->}[dr]_{\W (\phi_{\lambda })}                   & \lim\limits_{\POM} \W (R_\lambda ) \ar@{->}[d]^{\varphi}       \\
		&    \W (\lim R_\lambda ) }  
	\]
	
	We claim that $\varphi$ is an isomorphism. To see that it is an order-embedding, let $a,b\in \lim_{\POM} \W (R_\lambda )$ be such that $\varphi(a)\le \varphi(b)$ in $\W (\lim R_\lambda )$. Let $\lambda\in\Omega$ and $x,y \in M_\infty (R_\lambda )$ be such that $a=[\phi_\lambda(x)]$ and $b=[\phi_\lambda(y)]$. By the commutativity of the diagram above, we get
	\[
	\W (\phi_{\lambda})([x])\le \W (\phi_{\lambda})([y])
	\]
	and, therefore,  $\phi_\lambda (x)\precsim_1 \phi_\lambda (y)$ in $M_\infty (\lim R_\lambda )$.
	
	By the equational nature of $\precsim_1$ and the construction of the inductive limit in $\mathrm{Rings}^{ws}$, there exists $\mu\geq \lambda$ such that $\phi_{\mu ,\lambda }(x)\precsim_1\phi_{\mu ,\lambda }(y)$. This implies that $[\phi_{\mu ,\lambda }(x)]\leq [\phi_{\mu ,\lambda }(y)]$ in $\W (R_\mu)$. Thus, one has 
	\[
	a=[\phi_\lambda(x)]=[\phi_\mu(\phi_{\mu,\lambda}(x))]\leq [\phi_{\mu}(\phi_{\mu ,\lambda }(y))]=[\phi_{\lambda}(y)]=b,
	\]
	as desired.
	
	To check that $\varphi$ is also surjective, note that an element of $\W (\lim R_\lambda )$ is of the form $[\phi_{\lambda}(x)]$ for some $\lambda$ and $x\in M_\infty (R_\lambda )$. Therefore, one has
	\[
	[\phi_{\lambda}(x)] = \W (\phi_{\lambda})([x]) \in \varphi \left( \lim_{\POM} \W (R_\lambda ) \right) ,
	\]
	as required.
\end{proof}
We thus obtain a different proof of \autoref{thm:limitlambda} in the case of weakly $s$-unital rings.
\begin{corollary}\label{thm:ParCon}
	The functor $\Lambda \colon \mathrm{Rings}^{ws}\to \Cu$, $R\mapsto \Lambda (R)$, is continuous.
\end{corollary}
\begin{proof}
	Let $((R_{\lambda})_{\lambda\in\Omega},(\phi_{\mu,\lambda})_{\mu\geq\lambda})$ be a direct system in $\mathrm{Rings}^{ws}$. By \cite[Proposition 2.17]{AntAraBosPerVil23:CuRing}, we have that $\Lambda(R)\cong\Lambda_\sigma(\W(R))$ for any weakly $s$-unital ring. Using this at the first and the last step, and the isomorphisms from \autoref{prp:WSeqCont} at the second and third step, we obtain
	\[
	\Lambda (\lim R_\lambda)\! \cong  
	\Lambda_\sigma (\W (\lim R_\lambda))\! \cong 
	\Lambda_\sigma (\lim_{\rm POM} \W (R_\lambda))\! \cong 
	\lim_\Cu \Lambda_\sigma (\W (R_\lambda))\! \cong 
	\lim_\Cu \Lambda (R_\lambda).\qedhere
	\]
\end{proof}
\section{Continuity of the functor $\SCu$}\label{sec:ContinuitySCu}
In this final section we study the continuity of the functor $\SCu(\text{--})$. To this end, we first prove that the category $\SCu$ admits direct limits. Although, as shown in \autoref{thm:limitlambda} the first component $\Lambda(\text{--})$ is continuous for dense rings, $\SCu(\text{--})$ may not always be continuous in this case, as described in \autoref{exa:SCunotcont}. This is remedied by restricting to the class of dense left normal rings; see \autoref{thm:continuity-in-normalcase}.
\begin{theorem}\label{thm:DirLimEx}
	Direct systems $((S_{\lambda},W_{\lambda})_{\lambda\in\Omega},(\phi_{\mu,\lambda})_{\mu\geq\lambda})$ in $\SCu$ have a limit.
\end{theorem}
\begin{proof}
	Let $((S_{\lambda},W_{\lambda})_{\lambda\in\Omega},(\phi_{\mu,\lambda})_{\mu\geq\lambda})$ be a direct system in $\SCu$. By definition, we have a family $(S_{\lambda},W_{\lambda})_{\lambda\in\Omega}$ of objects in $\SCu$ indexed by a directed set $\Omega$ such that, for every pair $\mu\geq \lambda$ there exists a morphism $\phi_{\mu,\lambda}\colon (S_{\lambda},W_{\lambda})\to (S_{\mu},W_{\mu})$ in $\SCu$ such that $\phi_{\lambda ,\lambda }=\id$ and $\phi_{\mu,\lambda}\phi_{\lambda,\nu}=\phi_{\mu,\nu}$.
	
	Then, the induced system $((S_{\lambda})_{\lambda\in\Omega},(\phi_{\mu,\lambda})_{\mu\geq\lambda})$ is a direct system in $\Cu$ which, as  mentioned in \autoref{pgr:LimCu}, has a direct limit $S:=\lim_\lambda S_\lambda$. Denote by $\phi_\lambda$ the canonical maps $\phi_\lambda\colon S_\lambda\to\lim_\lambda S_\lambda$ given by the induced limit in $\Cu$.
		
	Let $W_0$ denote the union $\cup_\lambda \phi_\lambda (W_\lambda)$ in $S$. Note that, since $W_\lambda$ is a submonoid of $S_\lambda$  and each $\phi_{\mu ,\lambda}$ maps $W_\lambda$ to $W_\mu$, it follows that $W_0$ is a submonoid of $S$.
	
	By \autoref{prp:Lma54} (and the comments in \autoref{par:wincr}), we know that  every weakly increasing sequence in $S$ has a supremum. Now consider the set 
	\[
	W= \left\{
	w =\sup_nw_n\colon (w_n)\subseteq W_0\text{ is a weakly increasing sequence in } S 
	\right\}.
	\]
	Given any $v\in W_\lambda$, it follows that $\phi_\lambda (v)\in W$ by simply considering the constant sequence $(\phi_\lambda (v))_\lambda$. This shows that $\phi_\lambda (W_\lambda )\subseteq W$ for every $\lambda$, and thus also $W_0\subseteq W$.
	
	We claim that $(S,W)$ is the limit of the system $((S_\lambda, W_\lambda),\phi_{\mu,\lambda})$. To see this, let us first show that $(S,W)\in \SCu$. In other words, we need to prove that $W$ is a submonoid of $S$ closed under suprema of weakly increasing sequences.
	
	First, given $u,w\in W$, it follows from \autoref{par:wincr} that $u+w\in W$, and thus $W$ is a submonoid of $S$. Next, let $(w_n)\subseteq W$ be a weakly increasing sequence in $S$, and let $w\in S$ be its supremum. The proof of \autoref{prp:Lma54} yields a $\ll$-increasing sequence $(u_k)$ in $S$ satisfying (i)' and (ii)' as stated in said proof. In particular, by (ii)' we have that $w_n\leq \sup_k u_k$ for all $n$.
	
	By (i)', for each $k$ there is $n_k$ such that $u_k\ll w_n$ whenever $n\geq n_k$. We may also assume without loss of generality that the sequence $(n_k)_k$ is strictly increasing. Now, for $n<n_1$, we set $v_n=0$. Given $n$ such that $n_k\leq n<n_{k+1}$, use that $u_k\ll w_n$ for any such $n$ and the fact that $w_n\in W$, hence is the supremum of a weakly increasing sequence of elements in $W_0$ to find $v_n \in W_0$ such that
	\[
	u_k\ll v_n\leq w_n.
	\]
	
	By construction (and using also that $u_k$ is $\ll$-increasing), for every $k$ we have that $u_k\ll v_n$ whenever $n\geq n_k$. Further, $v_n\leq w_n\leq \sup_k u_k$ for every $n$. 
	
	Thus, again using the proof of \autoref{prp:Lma54}, we have that $(v_n)\subseteq W_0$ is a weakly increasing sequence. Moreover, it is clear from our construction that 
	\[
	\sup_k u_k = \sup_n v_n = \sup w_n=w.
	\]
	Therefore $w\in W$ and thus $(S,W)\in\SCu$.
	
	Finally, let us check that $(S,W)$ is the limit of $((S_{\lambda},W_{\lambda})_{\lambda\in\Omega},(\phi_{\mu,\lambda})_{\mu\geq\lambda})$. To this end, let $(T,Z)$ be an object of $\SCu$ and let $\{\psi_\lambda\}$ be a compatible family of morphisms $\psi_\lambda \colon (S_\lambda ,W_\lambda )\to (T,Z)$ in $\SCu$. Since $(S,\{\phi_{\lambda}\})$ is the direct limit of $((S_{\lambda})_{\lambda\in\Omega},(\phi_{\mu,\lambda})_{\mu\geq\lambda})$ in $\Cu$, it follows that there is a unique $\Cu$-morphism $\psi\colon S\to T$  such that $\psi_\lambda= \psi \circ \phi_\lambda$ for all $\lambda\in\Omega$.
	
	Now, let $w\in W$, and let $(w_n)\subseteq W_0$ be a weakly increasing sequence in $S$ with $w=\sup_n w_n$. Then, by \autoref{par:wincr}, $(\psi(w_n))$ is also a weakly increasing sequence in $T$ with $\psi(w)=\sup_n\psi(w_n)$. Furthermore 
	\[
	(\psi(w_n))\subseteq \psi(W_0)=\psi(\bigcup(\phi_\lambda(W_\lambda)))\subseteq\bigcup\psi(\phi_\lambda(W_\lambda))=\bigcup\psi_\lambda(W_\lambda)\subseteq Z.
	\]
	Since $Z$ is closed under suprema of weakly increasing sequences, we conclude that $\psi(w)\in Z$.	This implies that $\psi$ is a morphism in $\SCu$, as desired.
\end{proof}
We now proceed to show that $\SCu(\text{--})$ is in general not continuous. In the construction below note that, albeit similar, the ring $R$ is not the one used in \cite[Remark~4.8]{AntAraBosPerVil23:CuRing}.
\begin{example}
	\label{exa:SCunotcont}
	There exists a sequence of (unital commutative) rings $(R_n)$ and ring homomorphisms $f_n\colon R_n\to R_{n+1}$ such that $\SCu(\lim R_n)\not\cong\lim\limits_{\SCu}\SCu(R_n)$.	
\end{example}	
\begin{proof}
	Let $K$ be a field. Let $R_n$ be the ring $K[x_1,x_2,\ldots ,x_n]$ with commuting variables $x_1,x_2,\ldots , x_n$ subject to the relations $x_{i+1}x_i= x_i$ for $i=1,\ldots , n-1$. Let $R:=K[x_1,\ldots ]$ be the polynomial algebra in infinitely many commuting variables subject to the relations $x_{i+1}x_i = x_i$ for each $i\geq 1$. Clearly, one has $R=\lim_n R_n$, where the connecting maps and limit maps are given by the natural inclusions $\iota_{n+1,n}\colon R_n\to R_{n+1}$ and $\iota_n\colon R_n\to R$ respectively.
	
	Let $n\in\N$. Observe that each element $a$ of $R_n$ can be uniquely written in the form
	$$a= a_0 + x_1p_1(x_1)+ \cdots + x_n p_n(x_n),$$
	where $a_0\in K$, and $p_i(x_i)\in K[x_i]$. 
	For each $1\leq j\leq n$, we let $I_j$ the ideal $(x_r)$ of $R_n$ generated by $x_1,\dots,x_j$, and clearly $\{0\}\subseteq I_1\subseteq I_2\subseteq \ldots \subseteq I_n$. Note that, with respect to the above normal form of the elements of $R_n$, we have that $a\in I_j$ if and only if $a_0 = p_{j+1}= \cdots = p_n=0$.  
	
	We claim that, if $w=(w_i)_i\in \preS (R_n)$ with each $w_i$ in $M_{\infty}(I_n)$,  then $w=0$.	To see this suppose, by way of contradiction, that $w\ne 0$. 
	
	For each $1\leq r\leq n$, let $\pi_r\colon R_n\to K[x_r]$ be the homomorphism defined by 
	\[
	\pi_r(x_i)=
	\begin{cases}
		0, \text{ if } 1\leq i<r\\
		x_r, \text{ if } i=r\\
		1, \text{ if } r+1\leq i\leq n.
	\end{cases}
	\]
	Now, choose $r$ to be the smallest integer such that each $w_i$ belongs to $M_{\infty}(I_r)$. Therefore $\pi_r (w)= (\pi_r(w_1),\pi_r (w_2),\ldots )$ is a nonzero element in $\preS (K[x_r])$. Indeed there is an entry $a$ of $w_i$ for some $i\in \N$ such that 
	$$a= x_1p_1(x_1)+ \cdots + x_r p_r(x_r)\in I_r \setminus I_{r-1},$$
	which implies that $p_r\ne 0$. Hence $\pi_r (a) = x_rp_r (x_r) \ne 0$, showing that $\pi_r (w_i)\ne 0$. Hence $\pi_r (w)\ne 0$. Note that all the entries in each matrix $\pi_r (w_i)$ belong to the ideal of $K[x_r]$ generated by $x_r$. Hence by \autoref{Ideal_WR} $\pi_r(w)\in \mathcal S ((x_r))$ and we observed in \autoref{exa:IdNoBij} that $\mathcal S ((x_r)) =0$, hence we get a contradiction. Therefore our claim is proved. 
	
	Now write $(S,W)=\lim\limits_{\SCu} \SCu (R_n)$ (this limit exists by \autoref{thm:DirLimEx}). We want to show that the natural map $\Phi\colon(S,W)\to \SCu(R)=(\Lambda(R),\SR(R))$ is not an isomorphism. By the universal property of the inductive limit, if $\varphi_n\colon \SCu(R_n)\to (S,W)$ denote the limit maps, then $\Phi$ satisfies $\Phi\circ\varphi_n=\SCu(\iota_n)$ for each $n$. 
	
	Let $(x_n)\in\preS (R)$ be the sequence given by the commuting variables of $R$. To finish the proof, it is enough to show that $z:=[(x_n)]\in \SR (R)\setminus\Phi(W)$.
	
	Denote by $\pi \colon R\to K$  the homomorphism that sends all variables $x_n$ to $0$, and  let $\pi_n\colon R_n\to K$ be given by $\pi_n = \pi \circ \iota_{n}$.
	
	Suppose that $z\in \Phi(W)$. Then, by the proof of  \autoref{thm:DirLimEx}
	there is a  sequence $(w_n)\subseteq \cup_m \Lambda(\iota_{m}) (\SR (R_m))$, weakly increasing in $\Lambda(R)$, such that $z= \sup w_n$. In particular $\Lambda(\pi)(w_n)\leq\Lambda(\pi)(z)=0$ for all $n$, and thus 
	$$\Lambda (\pi) (w_n)= 0 \qquad  \text{ for all }  n .$$  
	For each $n$, choose $m=m(n)$ such that $w_n=\Lambda(\iota_m)(\tilde{w}_m)$ for $\tilde{w}_m\in\SR(R_m)$. Therefore, for all such $m$, we have 
	\[
	\Lambda(\pi_m)(\tilde{w}_m)=\Lambda(\pi\circ\iota_m)(\tilde{w}_m)=\Lambda(\pi)(w_n)=0.
	\] 	
	This means that, if we write $\tilde{w}_m= [(w^m_i)_i]$ with $(w^m_i)_i\in\preS (R_m)$, then $w^m_i\in M_\infty (\langle x_1,\ldots ,x_m\rangle )\subseteq M_\infty (R_m)$ for all $i$. But by the claim proved above we have $\tilde{w}_m=0$, whence also $w_n=0$ for all $n$. Thus, $z=\sup w_n=0$, which is a contradiction, because $z\ne 0$. Hence $z\in \SR (R)\setminus\Phi(W)$, as desired. 
\end{proof}
\begin{lemma}
	\label{lma:scu}
	Let $((S_{\lambda},W_{\lambda})_{\lambda\in\Omega},(\phi_{\mu,\lambda})_{\mu\geq\lambda})$ be a direct system in $\SCu$ with $W_\lambda\in \Cu$ for all $\lambda$. Then, 
	\[
	\varinjlim\limits_{\SCu}(S_\lambda,W_\lambda)=(\varinjlim\limits_{\Cu}S_\lambda,\varinjlim\limits_{\Cu}W_\lambda).
	\]
\end{lemma}
\begin{proof}
	Let $S=\varinjlim\limits_{\Cu}S_\lambda$ and $W=\varinjlim\limits_{\Cu}W_\lambda$.
	Let $(S,\tilde{W})$ be the limit in $\SCu$ of the system in the statement. By construction (see \autoref{thm:DirLimEx}) we have that $S=\varinjlim\limits_{\Cu}S_\lambda$. Denote by $f_\lambda\colon (S_\lambda, W_\lambda)\to (S,W)$ the natural maps, and by $\theta_\lambda\colon (S_\lambda,W_\lambda)\to (S,\tilde{W})$ the limit maps. Then, there is a (unique) $\SCu$-morphism $\Phi\colon (S,\tilde{W})\to (S,W)$ such that $\Phi\circ\theta_\lambda=f_\lambda$. Note that $\Phi_{|S}=\mathrm{id}_S$.
	
	By definition we already have that $\Phi(\tilde{W})\subseteq W$, hence it remains to show that $W\subseteq\Phi(\tilde{W})$. Let $w\in W$, and choose $w_n$ such that $w_{n}\ll w_{n+1}\ll w$ in $W$, which is possible since $W$ is by assumption a $\Cu$-semigroup. Using that $W=\varinjlim\limits_{\Cu}W_\lambda$ (see \autoref{pgr:LimCu}) we may find $\lambda_n$ and $y_n\in W_{\lambda_n}$ such that $w_n\leq f_{\lambda_n}(y_n)\leq w_{n+1}$. This implies that $w=\sup_n f_{\lambda_n}(y_n)$ and thus $w\in\Phi(\tilde{W})$.	
\end{proof}

\begin{theorem}\label{thm:continuity-in-normalcase}
	Let $((R_{\lambda})_{\lambda\in\Omega},(\phi_{\mu,\lambda})_{\mu\geq\lambda})$ be a direct system of dense, left normal rings. Then $\lim \SCu(R_\lambda) \cong \SCu (\lim R_\lambda)$.
\end{theorem}
\begin{proof} 
	Write $R=\lim _\lambda R_\lambda$, and denote by $\phi_\lambda\colon R_\lambda\to R$ the limit maps. We already know that $R$ is dense and left normal.
	
	By \autoref{thm:dense} we know that $\Lambda(R_\lambda)$ and $\SR(R_\lambda)$ are objects in $\Cu$, and by \autoref{thm:limitlambda}, we have that $\Lambda(R)=\varinjlim\limits_{\Cu}\Lambda(R_\lambda)$ as well as $\SR(R)=\varinjlim\limits_{\Cu}\SR(R_\lambda)$. Therefore we may apply \autoref{lma:scu} to conclude that $\varinjlim\limits_{\SCu}(\Lambda(R_\lambda),\SR(R_\lambda))=(\Lambda(R), \SR(R))$, as desired.
\end{proof}


\begin{thebibliography}{10}
	
	\bibitem{AntAraBosPerVil23:CuRing}
	R.~Antoine, P.~Ara, J.~Bosa, F.~Perera, and E.~Vilalta.
	\newblock The {C}untz semigroup of a ring.
	\newblock \emph{Selecta Math.}, to appear. arXiv:2307.07266 [math.RA], 2023.
	
	\bibitem{AntDadPerSan11}
	R.~Antoine, M.~Dadarlat, F.~Perera, and L.~Santiago.
	\newblock Recovering the {E}lliott invariant from the {C}untz semigroup.
	\newblock {\em Trans. Amer. Math. Soc.}, 366(6):2907--2922, 2014.
	
	\bibitem{APT-Memoirs2018}
	R.~Antoine, F.~Perera, and H.~Thiel.
	\newblock Tensor products and regularity properties of {C}untz semigroups.
	\newblock {\em Mem. Amer. Math. Soc.}, 251(1199):viii+191, 2018.
	
	\bibitem{APT-IMRN2020}
	R.~Antoine, F.~Perera, and H.~Thiel.
	\newblock Abstract bivariant {C}untz semigroups.
	\newblock {\em Int. Math. Res. Not. IMRN}, 2020(17):5342--5386, 2020.
	
	\bibitem{AAlgCol2004}
	P.~Ara.
	\newblock Rings without identity which are {M}orita equivalent to regular
	rings.
	\newblock {\em Algebra Colloq.}, 11(4):533--540, 2004.
	
	\bibitem{AraGol93}
	P.~Ara and D.~Goldstein.
	\newblock A solution of the matrix problem for {R}ickart C*-algebras.
	\newblock {\em Math. Nachr.}, 164:259--270, 1993.
	
	\bibitem{AraGooOmPar98}
	P.~Ara, K.~R. Goodearl, K.~C. O'Meara, and E.~Pardo.
	\newblock Separative cancellation for projective modules over exchange rings.
	\newblock {\em Israel J. Math.}, 105:105--137, 1998.
	
	\bibitem{AraPerPar2000}
	P.~Ara, E.~Pardo, and F.~Perera.
	\newblock The structure of countably generated projective modules over regular
	rings.
	\newblock {\em J. Algebra}, 226(1):161--190, 2000.
	
	\bibitem{APT2011}
	P.~Ara, F.~Perera, and A.~Toms.
	\newblock {$K$}-theory for operator algebras. {C}lassification of
	C*-algebras.
	\newblock In {\em Aspects of operator algebras and applications}, volume 534 of
	{\em Contemp. Math.}, pages 1--71. Amer. Math. Soc., Providence, RI, 2011.
	
\bibitem{BerDick78}
G.~M.~Bergman and W.~Dicks.
\newblock Universal derivations and universal ring constructions.
\newblock {\em Pacific J. Math.}, 79(2):293--337, 1978.
	
	
	\bibitem{Ciuperca2010}
	A.~Ciuperca, L.~Robert, and L.~Santiago.
	\newblock The {C}untz semigroup of ideals and quotients and a generalized
	Kasparov stabilization theorem.
	\newblock {\em J. Operator Theory}, 64(1):155--169, 2010.
	
	\bibitem{Coward2008}
	K.~Coward, G.~A.~Elliott, and C.~Ivanescu.
	\newblock The {C}untz semigroup as an invariant for C*-algebras.
	\newblock {\em J. Reine Angew. Math.}, 623:161--193, 2008.
	
	\bibitem{DickSch88}
	W.~Dicks and A.~H. Schofield.
	\newblock On semihereditary rings.
	\newblock {\em Comm. Algebra}, 16(6):1243--1274, 1988.
	
	\bibitem{Eff63}
	E.~G. Effros.
	\newblock Order ideals in a C*-algebra and its dual.
	\newblock {\em Duke Math. J.}, 30:391--411, 1963.
	
	\bibitem{FHK}
	A.~Facchini and F.~Halter-Koch.
	\newblock Projective modules and divisor homomorphisms.
	\newblock {\em J. Algebra Appl.}, 2(4):435--449, 2003.
	
	\bibitem{gardella2023semiprimeidealscalgebras}
	E.~Gardella, K.~Kitamura, and H.~Thiel.
	\newblock Semiprime ideals in C*-algebras, Preprint, arXiv:2311.17480 [math. OA], 2023.
	
	\bibitem{GarPer2022}
	E.~Gardella and F.~Perera.
	\newblock The modern theory of Cuntz semigroups of C*-algebras.
	\newblock EMS Surv. Math, Sci. (to appear) DOI: 10.4171/EMSS/84, 2024, 2022.
	
	\bibitem{GieHof+03Domains}
	G.~Gierz, K.~H. Hofmann, K.~Keimel, J.~D. Lawson, M.~Mislove, and D.~S. Scott.
	\newblock {\em Continuous lattices and domains}, volume~93 of {\em Encyclopedia
		of Mathematics and its Applications}.
	\newblock Cambridge University Press, Cambridge, 2003.
	
	\bibitem{HerPri2011}
	D.~Herbera and P.~P\v{r}\'{i}hoda.
	\newblock Big projective modules over noetherian semilocal rings.
	\newblock {\em J. Reine Angew. Math.}, 648:111--148, 2010.
	
	\bibitem{HerPri2014}
	D.~Herbera and P.~P\v{r}\'ihoda.
	\newblock Infinitely generated projective modules over pullbacks of rings.
	\newblock {\em Trans. Amer. Math. Soc.}, 366(3):1433--1454, 2014.
	
	\bibitem{Herbera2014}
	D.~Herbera and P.~P\v{r}\'{i}hoda.
	\newblock Reconstructing projective modules from its trace ideal.
	\newblock {\em J. of Algebra}, 416:25--57, 2014.
	
	\bibitem{JonTro74}
	S.~J\o{}ndrup and P.~J. Trosborg.
	\newblock A remark on pure ideals and projective modules.
	\newblock {\em Math. Scand.}, 35:16--20, 1974.
	
	\bibitem{Ped79}
	G.~K. Pedersen.
	\newblock {\em C*-algebras and their automorphism groups},
	volume~14 of {\em London Mathematical Society Monographs}.
	\newblock Academic Press, Inc. [Harcourt Brace Jovanovich, Publishers],
	London-New York, 1979. \newblock Second Edition, {\em Pure and Applied Mathematics (Amsterdam)}. 
	\newblock Academic Press, London, 2018.
	
	\bibitem{Ped86}
	G.~K. Pedersen.
	\newblock SAW*-algebras and corona C*-algebras, contributions
	to noncommutative topology.
	\newblock {\em J. Operator Theory}, 15(1):15--32, 1986.
	
	\bibitem{PedPet70}
	G.~K. Pedersen and N.~H. Petersen.
	\newblock Ideals in a C*-algebra.
	\newblock {\em Math. Scand.}, 27:193--204, 1970.
	
	\bibitem{ThiVil22DimCu}
	H.~Thiel and E.~Vilalta.
	\newblock Covering dimension of {C}untz semigroups.
	\newblock {\em Adv. Math.}, 394:44~p., Article No.~108016, 2022.
	
	\bibitem{Vas73}
	W.~V. Vasconcelos.
	\newblock Finiteness in projective ideals.
	\newblock {\em J. Algebra}, 25:269--278, 1973.
	
	\bibitem{Whthead80}
	J.~M. Whitehead.
	\newblock Projective modules and their trace ideals.
	\newblock {\em Comm. Algebra}, 8(19):1873--1901, 1980.
	
\end{thebibliography}
\end{document}